\documentclass[a4paper,12pt, abstract=true]{scrartcl}

\usepackage{german}
\usepackage[ansinew]{inputenc}
\usepackage[german,english,french,USenglish]{babel}
\usepackage{latexsym,amssymb,amsfonts,bm,bbm} 
\usepackage{theorem}
\usepackage{xcolor}
\usepackage{graphicx}
\usepackage{esint}
\usepackage{stmaryrd}
\usepackage[normalem]{ulem} 
\newtheorem{dummytheorem}{Dummy-Theorem}[section]
\newcommand{\proofendsign}{$\Box$} 
\newtheorem{definition}[dummytheorem]{Definition}
\newtheorem{lemma}[dummytheorem]{Lemma}
\newtheorem{theorem}[dummytheorem]{Theorem}
\newtheorem{proposition}[dummytheorem]{Proposition}

\newenvironment{proof}{{\noindent \bf Proof }}
 {{\hspace*{\fill}\proofendsign\par\bigskip}}

\theorembodyfont{\normalfont}
\newtheorem{remarknorm}[dummytheorem]{Remark}
\newtheorem{examplenorm}[dummytheorem]{Example}

\newcommand{\U}{\mathbb{U}}

\newcommand{\V}{\mathbf{V}}

\newcommand{\bE}{\mathbf{E}}

\newcommand{\N}{\mathbb{N}}
\newcommand{\Z}{\mathbb{Z}}

\newcommand{\R}{\mathbb{R}}

\newcommand{\E}{\mathbf{E}}

\newcommand{\F}{\mathbf{F}}

\newcommand{\D}{\mathbf{D}}
\newcommand{\bL}{\mathbf{L}}

\newcommand{\bB}{\mathbf{B}}

\newcommand{\bG}{\mathbf{G}}

\newcommand{\pr}{\mathbb{P}}
\newcommand{\ex}{\mathbb{E}}
\newcommand{\vari}{\mathbb{V}{\rm ar}}
\newcommand{\covi}{\mathbb{C}{\rm ov}}

\newcommand{\eins}{\mathbbm{1}}

\newcommand{\vatr}{{\rm V@R}}
\newcommand{\avatr}{{\rm AV@R}}

\newcommand{\argmin}{{\rm argmin}}
\newcommand{\OFP}{(\Omega,{\cal F},\pr)}

\newcommand{\cB}{{\cal B}}
\newcommand{\cF}{{\cal F}}
\newcommand{\cM}{{\cal M}}
\newcommand{\cR}{{\cal R}}

\newcommand{\cadlag}{c\`adl\`ag}

\newcommand{\cI}{{\cal I}}

\theorembodyfont{\rm}

\begin{document}


\title{Statistical inference for expectile-based risk measures}

\author{
Volker Kr\"atschmer\footnote{Faculty of Mathematics, University of Duisburg--Essen; {volker.kraetschmer@uni-due.de}}
\qquad\qquad
Henryk Zähle\footnote{Department of Mathematics, Saarland University; {zaehle@math.uni-sb.de}}}
\date{\small ~}

\maketitle

\begin{abstract}
Expectiles were introduced by Newey and Powell \cite{NeweyPowell1987} in the context of linear regression models. Recently, Bellini et al.~\cite{Bellinietal2014} revealed that expectiles can also be seen as reasonable law-invariant risk measures. In this article, we show that the corresponding statistical functionals are continuous w.r.t.\ the $1$-weak topology and suitably functionally differentiable. By means of these regularity results we can derive several properties such as consistency, asymptotic normality, bootstrap consistency, and qualitative robustness of the corresponding estimators in nonparametric and parametric statistical models.
\end{abstract}





\newpage

\section{Introduction}\label{Introduction}

Let $(\Omega,{\cal F},\pr)$ be an atomless probability space and use $L^p=L^p(\Omega,{\cal F},\pr)$ to denote the usual $L^p$-space. The $\alpha$-expectile of $X\in L^2$, with $\alpha\in(0,1)$, can uniquely be defined by
\begin{eqnarray}
    \rho_\alpha(X)
    & := & \argmin_{m\in\R}\,\big\{\alpha\,\ex\big[\big((X - m)^{+}\big)^{2}\big]+(1 - \alpha)\,\ex\big[\big((m - X)^{+}\big)^{2}\big]\big\} \nonumber\\
    & = & \argmin_{m\in\R}\,\ex[V_\alpha(X-m)] \label{def expectiles-based shortfall risk measure - alternative}
\end{eqnarray}
(see Proposition 1 and Example 4 in \cite{Bellinietal2014}), where
$$
    V_\alpha(x)
    :=
    \left\{
    \begin{array}{rll}
        \alpha\,x^2 & , & x\ge 0\\
        (1-\alpha)\,x^2 & , & x<0
    \end{array}
    \right.
    ,\qquad x\in\R.
$$
Expectiles were introduced by Newey and Powell \cite{NeweyPowell1987} in the context of linear regression models. On the one hand, (\ref{def expectiles-based shortfall risk measure - alternative}) generalizes the {\em expec}tation of $X$ which coincides with $\rho_\alpha(X)$ when specifically $\alpha=1/2$. On the other hand, (\ref{def expectiles-based shortfall risk measure - alternative}) is similar to the $\alpha$-quan{\em tile} of $X$ which can be obtained by replacing $x^2$ by $|x|$ in the definition of $V_\alpha$. This motivates the name {\em $\alpha$-expectile}.

For every $X\in L^2$ the mapping $m\mapsto\ex[V_\alpha(X-m)]$ is convex and differentiable with derivative given by $m\mapsto -2\U_\alpha(X)(m)$, where
\begin{equation}\label{Def Mapping bb U alpha}
    \mathbb{U}_\alpha(X)(m):=\ex[U_\alpha(X-m)],\qquad m\in\R
\end{equation}
with
$$
    U_\alpha(x)
    :=
    \left\{
    \begin{array}{rll}
        \alpha\,x & , & x\ge 0\\
        (1-\alpha)\,x  & , & x<0
    \end{array}
    \right.
    ,\qquad x\in\R.
$$
Moreover, for $X\in L^1$ the mapping $m\mapsto\U_\alpha(X)(m)$ is well defined and bijective; cf.\ Lemma \ref{representation int by parts - mathbb} (Appendix \ref{expectiles as risk measures}). These observations together imply that for $X\in L^2$ the $\alpha$-expectile admits the representation
\begin{equation}\label{def expectiles-based shortfall risk measure}
    \rho_\alpha(X) = \mathbb{U}_\alpha(X)^{-1}(0),
\end{equation}
where $\mathbb{U}_\alpha(X)^{-1}$ denotes the inverse function of $\mathbb{U}_\alpha(X)$. In particular, (\ref{def expectiles-based shortfall risk measure}) can be used to define a map $\rho_\alpha:L^1\to\R$ which is compatible with (\ref{def expectiles-based shortfall risk measure - alternative}). For every $X\in L^1$ the value in (\ref{def expectiles-based shortfall risk measure}) will be called the corresponding $\alpha$-expectile.

Recently, Bellini~et~al.~\cite{Bellinietal2014} revealed that expectiles can be also seen as reasonable risk measures when $1/2\le\alpha<1$. In Proposition 6 in \cite{Bellinietal2014}, they prove that the map $\rho_\alpha:L^2\to\R$ provides a coherent risk measure if (and only if) $1/2\le\alpha<1$. Recall that a map $\rho:{\cal X}\to\R$, with ${\cal X}$ a subspace of $L^0$, is said to be a coherent risk measure if it is
\begin{itemize}
    \item monotone: $\rho(X_1)\le\rho(X_2)$ for all $X_1,X_2\in{\cal X}$ with $X_1\le X_2$,
    \item cash-invariant: $\rho(X+m)=\rho(X)+m$ for all $X\in{\cal X}$ and $m\in\R$,
    \item subadditive: $\rho(X_1+X_2)\le\rho(X_1)+\rho(X_2)$ for all $X_1,X_2\in{\cal X}$,
    \item positively homogenous: $\rho(\lambda X)=\lambda\,\rho(X)$ for all $X\in{\cal X}$ and $\lambda\ge 0$.
\end{itemize}
It is shown in the Appendix \ref{expectiles as risk measures} (Proposition \ref{Expectiles as coherent RM}) that even the map $\rho_\alpha:L^1\to\R$ provides a coherent risk measure if (and only if) $1/2\le\alpha<1$. For $0<\alpha<1/2$ the map $\rho_\alpha:L^1\to\R$ is at least monotone, cash-invariant, and positively homogeneous. For this reason we will henceforth refer to $\rho_\alpha:L^1\to\R$ as {\em expectile-based risk measure} at level $\alpha\in(0,1)$. It is worth mentioning that $\rho_\alpha$ already appeared implicitly in an earlier paper by Weber \cite{Weber2006}. As Ziegel \cite{Ziegel2014} pointed out that $\rho_\alpha$ satisfies a particularly desirable property of risk measures in the context of backtesting, $\rho_\alpha$ attracted special attention in the field of monetary risk measurement in the last few years \cite{AcerbiSzekely2014,BelliniDiBernardino2015,Bellinietal2014,Delbaen2013,Emmeretal2015,Ziegel2014}. For pros and cons of expectile-based risk measures and of other standard risk measures see, for instance, the discussions by Acerbi and Szekely \cite{AcerbiSzekely2014}, Bellini and Di Bernardino \cite{BelliniDiBernardino2015}, and Emmer et al.\ \cite{Emmeretal2015}.



This article is concerned with the statistical estimation of expectile-based risk measures. The goal is the estimation of $\rho_\alpha(X)$ for some $X\in L^1$ with unknown distribution function $F$. Let $\F_1$ be the class of all distribution functions on $\R$ satisfying $\int|x|\,dF(x)<\infty$. Note that $\F_1$ coincides with the set of the distribution functions of all elements of $L^1$, because the underlying probability space was assumed to be atomless. Also note that $F\in\F_{1}$ if and only if $\int_{-\infty}^{0}F(x)\,dx<\infty$ and $\int_{0}^{\infty}(1 - F(x))\,dx < \infty$ hold. Since $\rho_\alpha$ is law-invariant (i.e.\ $\rho_\alpha(X_1)=\rho_\alpha(X_2)$ when $\pr\circ{X_1}^{-1}=\pr\circ{X_2}^{-1}$), we may associate with $\rho_\alpha$ a statistical functional ${\cal R}_\alpha:\F_1\rightarrow\R$ via
\begin{equation}\label{definition of risk functional}
    {\cal R}_\alpha(F_X)\,:=\,\rho_\alpha(X),\qquad X\in L^1,
\end{equation}
where $F_X$ denotes the distribution function of $X$. That is,
\begin{equation}\label{def risk functional}
    {\cal R}_\alpha(F)\,=\,{\cal U}_\alpha(F)^{-1}(0) \qquad\mbox{for all $F\in\F_1$},
\end{equation}
where
\begin{equation}\label{Def Mapping cal U alpha}
    {\cal U}_\alpha(F)(m)\,:=\,\int U_\alpha(x-m)\,dF(x),\qquad m\in\R.
\end{equation}
Then, if $\widehat F_n$ is a reasonable estimator for $F$, 
the plug-in estimator ${\cal R}_\alpha(\widehat F_n)$ is typically a reasonable estimator for $\rho_\alpha(X)={\cal R}_\alpha(F)$.

In a nonparametric framework, a canonical example for $\widehat F_n$ is the empirical distribution function
\begin{equation}\label{def edf}
    \widehat F_n\,:=\,\frac{1}{n}\sum_{i=1}^{n}\eins_{[X_i,\infty)}
\end{equation}
of $n$ identically distributed random variables $X_1,\ldots,X_n$ drawn according to $F$. In this case we have
\begin{equation}\label{def z estimator}
    {\cal R}_\alpha(\widehat F_n)\,=\,{\cal U}_\alpha(\widehat F_n)^{-1}(0)\,=\,\mbox{unique solution in $m$ of }\sum_{i=1}^nU_\alpha(X_i-m)= 0.
\end{equation}
That is, the plug-in estimator is nothing but a simple Z-estimator (M-estimator). For Z-estimators (M-estimators) there are several results concerning consistency and the asymptotic distribution in the literature. A classical reference is Huber's seminal paper \cite{Huber1967}; see also standard textbooks as \cite{HuberRonchetti2009,Serfling1980,van der Vaart 1998,van der Vaart Wellner 1996}. Recently Holzmann and Klar \cite{HolzmannKlar2015} used results of Arcones \cite{Arcones2000} and Van der Vaart \cite{van der Vaart 1998} to derive asymptotic properties of the Z-estimator in (\ref{def z estimator}). They restricted their attention to i.i.d.\ observations but allowed for observations without finite second moment.

On the other hand, even in the nonparametric setting the estimator $\widehat F_n$ may differ from the empirical distribution function so that the plug-in estimator need not be a $Z$-estimator. See, for instance, Section 3 in \cite{BeutnerZaehle2010} for estimators $\widehat F_n$ being different from the empirical distribution function. Also, in a parametric setting the estimator $\widehat F_n$ will hardly be the empirical distribution function. For these reasons, we will consider a suitable linearization of the functional ${\cal R}_\alpha$ in order to be in the position to derive several asymptotic properties of the plug-in estimator ${\cal R}_\alpha(\widehat F_n)$ in as many as possible situations.

Linearizations of Z-functionals have been considered before, for instance, by Clarke \cite{Clarke1983,Clarke1986}. However, these results do not cover the particular Z-functional ${\cal R}_\alpha$, because the function $U_\alpha$ is unbounded. By using the concept of {\em quasi}\,-Hadamard differentiability as well as the corresponding functional delta-method introduced by Beutner and Z\"ahle \cite{BeutnerZaehle2010,BeutnerZaehle2015} we will overcome the difficulties with the unboundedness of $U_\alpha$. Quasi-Hadamard differentiability of the functional ${\cal R}_\alpha$ will in particular admit some bootstrap results for the plug-in estimator ${\cal R}_\alpha(\widehat F_n)$ when $\widehat F_n$ is the empirical distribution function of $X_1,\ldots,X_n$.

It is worth mentioning that Heesterman and Gill \cite{HeestermanGill1992} also considered a linearization approach to $Z$-estimators. However (boiled down to our setting) they did not consider a linearization of the functional ${\cal R}_\alpha$ (to be evaluated at the estimator $\widehat F_n$ of $F$) but only of the functional that provides the unique zero of a strictly decreasing and continuous function tending to $\pm\infty$ as its argument tends to $\mp\infty$ (as the function ${\cal U}_\alpha(\widehat F_n)$). To some extent this approach is less flexible than our approach. Especially parametric estimators cannot be handled by this approach without further ado.

The rest of this article is organized as follows. In Section \ref{section regularity} we will establish a certain continuity and the above-mentioned differentiability of the functional ${\cal R}_\alpha$. In Sections \ref{NonParamEstim}--\ref{ParamEstim} we will apply the results of Section \ref{section regularity} to the nonparametric and parametric estimation of ${\cal R}_\alpha(F)$. In Section \ref{proof of main theorem} we will prove the main result of Section \ref{section regularity}, and in Section \ref{Sec Remaining proofs} we will verify two examples and a lemma  presented in Sections \ref{NonParamEstim}--\ref{ParamEstim}. The Appendix provides some auxiliary results. In particular, in Section \ref{appendix QHD and FDM} of the Appendix we formulate a slight generalization of the functional delta-method in the form of Beutner and Z\"ahle \cite{BeutnerZaehle2015}.


\section{Regularity of the functional ${\cal R}_\alpha$}\label{section regularity}

In this section we investigate the functional ${\cal R}_\alpha:\F_1\to\R$ defined in (\ref{def risk functional}) for continuity and differentiability. We equip $\F_1$ with the $1$-weak topology. This topology is defined to be the coarsest topology for which the mappings $\mu\mapsto\int f\,dF$, $f\in {\cal C}_{1}$, are continuous, where ${\cal C}_{1}$ is the set of all continuous functions $f:\R\to\R$ with $|f(x)|\le C_f(1+|x|)$ for all $x\in\R$ and some finite constant $C_f>0$. A sequence $(F_n)\subseteq\F_1$ converges $1$-weakly to some $F_0\in\F_1$ if and only if $\int f\,dF_n\to\int f\,dF_0$ for all $f\in{\cal C}_1$; cf.\ Lemma 3.4 in \cite{Kraetschmeretal2012}.  The set $\F_1$ can obviously be identified with the set of all Borel probability measures $\mu$ on $\R$ satisfying $\int|x|\,\mu(dx)<\infty$. In this context the $1$-weak topology is sometimes referred to as $\psi_1$-weak topology; see, for instance, \cite{Kraetschmeretal2014}. But for our purposes it is more convenient to work with the $\F_1$-terminology.

Let $\bL_0$ be the space of all Borel measurable functions $v:\R\rightarrow\R$ modulo the equivalence relation of $\ell$-almost sure identity. Note that $\F_1\subseteq\bL_0$, and let $\bL_1\subseteq\bL_0$ be the subspace of all $v\in\bL_0$ for which
\begin{equation}\label{Def Wasserstein norm}
    \|v\|_{1,\ell}:=\int|v(x)|\,\ell(dx)
\end{equation}
is finite. Here, and henceforth, $\ell$ stands for the Borel Lebesgue measure on $\R$. Note that $F_{1} - F_{2}\in \bL_{1}$ for $F_{1}, F_{2}\in \F_{1}$. It is well-known that $\|\cdot\|_{1,\ell}:\bL_1\rightarrow\R_+$ provides a complete and separable norm on $\bL_1$ and that
$$
    d_{\mbox{\scriptsize{\rm W}},1}(F_1,F_2):=\|F_1-F_2\|_{1,\ell}
$$
defines the Wasserstein-$1$ metric $d_{\mbox{\scriptsize{\rm W}},1}:\F_1\times\F_1\rightarrow\R_+$ on $\F_1$. Also note that $d_{\mbox{\scriptsize{\rm W}},1}$ metrizes the $1$-weak topology on $\F_1$; cf.\ Remark 2.9 in \cite{Kraetschmeretal2014}.


\subsection{Continuity}\label{section regularity - continuity}

Since the Wasserstein-$1$ metric $d_{\mbox{\scriptsize{\rm W}},1}$ metrizes the $1$-weak topology on $\F_1$, the following theorem is an immediate consequence of a recent result by Bellini et al.~\cite[Theorem 10]{Bellinietal2014}.

\begin{theorem}\label{1 weak continuity of R alpha}
The functional ${\cal R}_\alpha:\F_1\to\R$ is continuous for the $1$-weak topology.
\end{theorem}

Theorem \ref{1 weak continuity of R alpha} can also be obtained by combining Theorem 4.1 in \cite{CheriditoLi2009} with the Representation theorem 3.5 in \cite{Kraetschmeretal2014}. Indeed, these two theorems together imply that the risk functional associated with any law-invariant coherent risk measure on $L^1$ is $1$-weakly continuous. For $1/2\le\alpha<1$ the functional ${\cal R}_\alpha$ itself is derived from a law-invariant coherent risk measure (see Proposition \ref{Expectiles as coherent RM} in Appendix \ref{expectiles as risk measures}). So it is $1$-weakly continuous. For $0<\alpha<1/2$ the map $\check{\rho}_\alpha:L^1\rightarrow\R$ defined by $\check\rho_\alpha(X):=-\rho_\alpha(-X)$ provides a law-invariant coherent risk measure (cf.\ Proposition \ref{Expectiles as coherent RM} in
Appendix \ref{expectiles as risk measures}), so that the associated statistical functional $\check{{\cal R}}_{\alpha}: \F_{1}\mapsto\R$, $\check{{\cal R}}_\alpha(F)=-{\cal R}_\alpha(\check F)$, is $1$-weakly continuous. Here $\check F$ stands for the distribution function derived from $F$ via $\check F(x):=1-F((-x)-)$. Since for any sequence $(F_n)_{n\in\N_0}\subseteq\F_1$, $F_n\to F_0$ $1$-weakly if and only if $\check F_n\to \check F_0$ $1$-weakly, it follows that also the functional ${\cal R}_\alpha$ is $1$-weakly continuous.

By the $1$-weak continuity of ${\cal R}_\alpha$ we are in the position to easily derive strong consistency of the plug-in estimator ${\cal R}_\alpha(\widehat F_n)$ for ${\cal R}_\alpha(F)$ in several situations; see Sections \ref{section strong consistency} and \ref{section strong consistency - param}.


\subsection{Differentiability}\label{section regularity - differentiability}

We will use the notion of quasi-Hadamard differentiability introduced in \cite{BeutnerZaehle2010,BeutnerZaehle2015}. Quasi-Hadamard differentiability is a slight (but useful) generalization of the conventional tangential Hadamard differentiability. The latter is commonly acknowledged to be a suitable notion of differentiability in the context of the functional delta-method (see e.g.\ the bottom of p.\,166 in \cite{HeestermanGill1992}), and it was shown in \cite{BeutnerZaehle2010,BeutnerZaehle2015} that the former is still strong enough to obtain a functional delta-method. Let $\bL_1$ be equipped with the norm $\|\cdot\|_{1,\ell}$. 

\begin{definition}\label{definition quasi hadamard - L 1}
Let ${\cal R}:\F_1\to\R$ be a map and $\bL_1^0$ be a subset of $\bL_1$. Then ${\cal R}$ is said to be quasi-Hadamard differentiable at $F\in\F_1$ tangentially to $\bL_1^0\langle\bL_1\rangle$ if there exists a continuous map $\dot{\cal R}_F:\bL_1^0\to\R$ such that
\begin{eqnarray}\label{def eq for HD}
    \lim_{n\to\infty}\Big|\dot{\cal R}_F(v)-\frac{{\cal R}(F+\varepsilon_nv_n)-{\cal R}(F)}{\varepsilon_n}\Big|\,=\,0
\end{eqnarray}
holds for each triplet $(v,(v_n),(\varepsilon_n))$, with $v\in\bL_1^0$, $(\varepsilon_n)\subseteq(0,\infty)$ satisfying $\varepsilon_n\to 0$, $(v_n)\subseteq\bL_1$ satisfying $\|v_n-v\|_{1,\ell}\to 0$ as well as $(F+\varepsilon_nv_n)\subseteq\F_1$. In this case the map $\dot{\cal R}_F$ is called quasi-Hadamard derivative of ${\cal R}$ at $F$ tangentially to $\bL_1^0\langle\bL_1\rangle$.
\end{definition}

Note that even when $\bL_1^0=\bL_1$, quasi-Hadamard differentiability of ${\cal R}$ at $F$ tangentially to $\bL_1\langle\bL_1\rangle$ is not the same as Hadamard differentiability of ${\cal R}$ at $F$ tangentially to $\bL_1$ (with $\bL_0$ regarded as the basic linear space containing both $\F_1$ and $\bL_1$). Indeed, $\|\cdot\|_{1,\ell}$ does not impose a norm on all of $\bL_0$ (but only on $\bL_1$), so that Hadamard differentiability w.r.t.\ the norm $\|\cdot\|_{1,\ell}$ is not defined.

\begin{theorem}\label{main theorem}
Let $F\in\F_1$ and assume that it is continuous at ${\cal R}_\alpha(F)$. Then the functional ${\cal R}_\alpha:\F_1\to\R$ is quasi-Hadamard differentiable at $F$ tangentially to $\bL_1\langle\bL_1\rangle$ with linear quasi-Hadamard derivative $\dot {\cal R}_{\alpha;F}:\bL_1\rightarrow\R$ given by
\begin{equation}\label{def of limit of emp process}
    \dot{\cal R}_{\alpha;F}(v)\,:=\,-\frac{(1-\alpha)\int_{(-\infty,0)}v(x+{\cal R}_\alpha(F))\,\ell(dx)\,+\,\alpha\int_{(0,\infty)} v(x+{\cal R}_\alpha(F))\,\ell(dx)}{ (1-2\alpha)F({\cal R}_\alpha(F))\,+\,\alpha}\,.\nonumber  
\end{equation}
\end{theorem}

Note that $(1-2\alpha)F({\cal R}_\alpha(F))\,+\,\alpha = (1-\alpha)F({\cal R}_\alpha(F))+\alpha(1-F({\cal R}_\alpha(F))) > 0$ holds so that the denominator in (\ref{def of limit of emp process}) is strictly positive. Also note that quasi-Hadamard differentiability is already known form Theorem 2.4 in \cite{Kraetschmeretal2014b}. However, in \cite{Kraetschmeretal2014b} the derivative was not specified explicitly. The proof of Theorem \ref{main theorem} can be found in Section \ref{proof of main theorem}.

\begin{remarknorm}\label{main theorem - remark}
As a direct consequence of Theorem \ref{main theorem} we obtain that the functional ${\cal R}_\alpha$ is also quasi-Hadamard differentiable at $F$ (being continuous at ${\cal R}_\alpha(F)$) tangentially to any subspace of $\bL_1$ that is equipped with a norm being at least as strict as the norm $\|\cdot\|_{1,\ell}$.
{\hspace*{\fill}$\Diamond$\par\bigskip}
\end{remarknorm}

\begin{examplenorm}\label{main theorem - example}
To illustrate Remark \ref{main theorem - remark}, let $\phi:\R\rightarrow[1,\infty)$ be a continuous function that is non-increasing on $(-\infty,0]$ and non-decreasing on $[0,\infty)$. Let $\F_\phi$ be the set of all distribution functions $F$ on $\R$ for which $\|F-\eins_{[0,\infty)}\|_\phi<\infty$, where $\|v\|_\phi:=\sup_{x\in\R}|v(x)|\phi(x)$. Let $\D$ be the space of all bounded \cadlag\ functions on $\R$ and $\D_\phi$ be the subspace of all $v\in\D$ satisfying $\|v\|_\phi<\infty$ and $\lim_{|x|\to\infty} |v(x)|=0$. If $C_\phi:=\int1/\phi\,d\ell<\infty$, then $\D_{\phi}\subseteq\bL_1$ and $\F_\phi\subseteq\F_1$. On the space $\D_\phi$ the norm $\|\cdot\|_\phi$ is stricter than $\|\cdot\|_{1,\ell}$, because
\begin{equation}\label{main theorem - corollary - proof - 10}
    \|v\|_{1,\ell}=\int|v(x)|\,\ell(dx)\le C_\phi\|v\|_\phi\qquad\mbox{for every }v\in\D_\phi.
\end{equation}
Therefore ${\cal R}_\alpha$ is also quasi-Hadamard differentiable at $F$ tangentially to $\D_\phi\langle\D_\phi\rangle$ with linear quasi-Hadamard derivative $\dot {\cal R}_{\alpha;F}:\D_\phi\rightarrow\R$ given by (\ref{def of limit of emp process}) restricted to $v\in\D_\phi$, where $\D_\phi$ is equipped with the norm $\|\cdot\|_\phi$.
{\hspace*{\fill}$\Diamond$\par\bigskip}
\end{examplenorm}

The established quasi-Hadamard differentiability of ${\cal R}_\alpha$ brings us in the position to easily derive results on the asymptotics of ${\cal R}_\alpha(\widehat F_n)$; see Sections \ref{section asymptotic distribution}--\ref{section bootstrapping the asymptotic distribution} and \ref{section asymptotic distribution - param}. In Section \ref{section asymptotic distribution} we combine Theorem \ref{main theorem} with a central limit theorem (by Dede \cite{Dede2009}; cf.\ Theorem \ref{CLT by Dede} below) for the empirical process in the space $(\bL_1,\|\cdot\|_{1,\ell})$ in order to obtain the asymptotic distribution of ${\cal R}_\alpha(\widehat F_n)$ in a rather general nonparametric setting. In view of Example \ref{main theorem - example} one can alternatively use central limit theorems for the empirical process in the space $(\D_\phi,\|\cdot\|_{\phi})$ to obtain the asymptotic distribution of ${\cal R}_\alpha(\widehat F_n)$. See, for instance, Examples 4.4--4.5 in \cite{BeutnerZaehle2015} as well as references cited there.


\section{Nonparametric estimation of ${\cal R}_\alpha(F)$}\label{NonParamEstim}

In this section we consider nonparametric statistical models. We will always assume that the sequence of observations $(X_i)$ is a strictly stationary sequence of real-valued random variables. In addition we will mostly assume that $(X_i)$ is ergodic; see Section 6.1 and 6.7 in \cite{Breiman} for the definition of a strictly stationary and ergodic sequence. Recall that every sequence of i.i.d.\ random variables is strictly stationary and ergodic. Moreover a strictly stationary sequence is ergodic when it is mixing in the ergodic sense, and it is mixing in the ergodic sense when it is $\alpha$-mixing; see Section 2.5 in \cite{Bradley2005}. For illustration, also note that many GARCH processes are strictly stationary and ergodic; cf.\ \cite{Boussamaetal2011,Nelson1990}.

Throughout this section the estimator for the marginal distribution function $F$ of $(X_i)$ is assumed to be the empirical distribution function $\widehat F_n$ of $X_1,\ldots,X_n$ as defined in (\ref{def edf}). Note that the mapping $\Omega\rightarrow\F_1$, $\omega\mapsto\widehat F_n(\omega,\cdot)$, is $({\cal F},{\cal B}(\F_1))$-measurable for the Borel $\sigma$-algebra ${\cal B}(\F_1)$ on $(\F_1,d_{\mbox{\scriptsize{\rm W}},1})$, because the mapping $\R^n\rightarrow\F_1$, $(x_{1},\ldots,x_{n})\mapsto \frac{1}{n}\sum_{i=1}^{n}\eins_{[x_{i},\infty)}$, is $(\|\cdot\|,d_{\mbox{\scriptsize{\rm W}},1})$-continuous. Hence by continuity of $\cR_{\alpha}$ w.r.t.\ $d_{\mbox{\scriptsize{\rm W}},1}$, it follows that $\cR_{\alpha}(\widehat{F}_{n})$ is a real-valued random variable on $(\Omega,{\cal F},\pr)$.


\subsection{Strong consistency}\label{section strong consistency}


For $1/2\le\alpha<1$ the following theorem is a direct consequence of Theorem 2.6 in \cite{Kraetschmeretal2014}. In the general case, Theorem \ref{1 weak continuity of R alpha} ensures that one can follow the lines in the proof of Theorem 2.6 in \cite{Kraetschmeretal2014} to obtain the assertion of Theorem \ref{Consistency thm}; we omit the details.

\begin{theorem}\label{Consistency thm}
Let $(X_i)$ be a strictly stationary and ergodic sequence of $L^1$-random variables on some probability space $(\Omega,{\cal F},\pr)$, and denote by $F$ the distribution function of the $X_i$. Let $\widehat F_n$ be the empirical distribution function of $X_1,\ldots,X_n$ as defined in (\ref{def edf}). Then the plug-in estimator ${\cal R}_\alpha(\widehat F_n)$ is strongly consistent for ${\cal R}_\alpha(F)$ in the sense that
$$
    {\cal R}_\alpha(\widehat F_n)\to{\cal R}_\alpha(F)\qquad\pr\mbox{-a.s.}
$$
\end{theorem}


If $X_1,X_2,\ldots$ are i.i.d.\ random variables, then strong consistency can also be obtained from classical results on Z-estimators as, for example, Lemma A in Section 7.2.1 of \cite{Serfling1980}. Moreover, it was shown recently by Holzmann and Klar \cite[Theorem 2]{HolzmannKlar2015} that in the i.i.d.\ case one even has $\sup_{\alpha\in[\alpha_\ell,\alpha_u]}|{\cal R}_\alpha(\widehat F_n)\to{\cal R}_\alpha(F)|\rightarrow 0$ $\pr$-a.s.\ for any $\alpha_\ell,\alpha_u\in(0,1)$ with $\alpha_\ell<\alpha_u$.


\subsection{Asymptotic distribution}\label{section asymptotic distribution}

Dedecker and Prieur \cite{DedeckerPrieur2005} introduced the following dependence coefficients for a strictly stationary sequence of real-valued random variables $(X_i)\equiv (X_i)_{i\in\N}$ on some probability space $(\Omega,{\cal F},\pr)$:
\begin{eqnarray}
    \widetilde\phi(n) & := & \sup_{k\in\N}\,\sup_{x\in\R}\,\|\,\pr[X_{n+k}\in(-\infty,x]|{\cal F}_1^k](\cdot)-\pr[X_{n+k}\in(-\infty,x]]\,\|_\infty,\label{tilde phi mc}\\
    \widetilde\alpha(n) & := & \sup_{k\in\N}\,\sup_{x\in\R}\,\|\,\pr[X_{n+k}\in (-\infty,x]|{\cal F}_1^k](\cdot)-\pr[X_{n+k}\in(-\infty,x]]\,\|_1. \label{tilde alpha mc}
\end{eqnarray}
Here ${\cal F}_1^k:=\sigma(X_1,\ldots,X_k)$ and $\|\cdot\|_p$ denotes the usual $L^p$-norm on $L^p=L^p(\Omega,{\cal F},\pr)$, $p\in[1,\infty]$. Note that by Proposition 3.22 in \cite{Bradley2002} the usual $\phi$- and $\alpha$-mixing coefficients $\phi(n)$ and $\alpha(n)$ can be represented as in (\ref{tilde phi mc})--(\ref{tilde alpha mc}) with $\sup_{x\in\R}$ and $(-\infty,x]$ replaced by $\sup_{A\in{\cal B}(\R)}$ and $A$, respectively. In particular, $\widetilde\phi(n)\le\phi(n)$ and $\widetilde\alpha(n)\le\alpha(n)$. It is worth mentioning that in \cite{DedeckerPrieur2005} the starting point is actually a strictly stationary sequence of random variables indexed by $\Z$ (rather than $\N$) and that therefore the definitions of the above dependence coefficients are slightly different. However, it is discussed in detail in the Appendix \ref{dependence coefficients} that any strictly stationary sequence $(X_i)\equiv(X_i)_{i\in\N}$ can be extended to a strictly stationary sequence $(Y_{i})_{i\in\Z}$ satisfying $\widetilde{\phi}(n) = \overline{\phi}(n)$ and $\widetilde{\alpha}(n) = \overline{\alpha}(n)$, where $\overline{\phi}(n)$ and $\overline{\alpha}(n)$ are the dependence coefficients of $(Y_i)_{i\in\Z}$ as originally introduced in \cite{DedeckerPrieur2005}. It is also discussed in the Appendix \ref{dependence coefficients} that if $(X_i)$ is in addition ergodic, then $(Y_{i})_{i\in\Z}$ is ergodic too.

Let us denote by $Q_{|X_1|}$ the \cadlag\ inverse of the tail function $x\mapsto\pr[|X_1|>x]$. Let us write ${\rm N}_{0,s^2}$ for the centered normal distribution with variance $s^2$. Moreover, let us use $\leadsto$ to denote convergence in distribution.

\begin{theorem}\label{main theorem coroll}
Let $(X_i)$ be a strictly stationary and ergodic sequence of real-valued random variables on some probability space $(\Omega,{\cal F},\pr)$. Denote by $F$ the distribution function of the $X_i$, and assume that $F$ is continuous at ${\cal R}_\alpha(F)$ and that $\int\sqrt{F(1-F)}\,d\ell<\infty$ (in particular $F\in\F_1$). Let $\widehat F_n$ be the empirical distribution function of $X_1,\ldots,X_n$ as defined in (\ref{def edf}). Finally assume that one of the following two conditions holds:
\begin{equation}\label{main theorem coroll - eq - 10}
    \sum_{n\in\N}n^{-1/2}\,\widetilde\phi(n)^{1/2}<\infty,
\end{equation}
\begin{equation}\label{main theorem coroll - eq - 20}
    \sum_{n\in\N}n^{-1/2}\int_{(0,\widetilde\alpha(n))}Q_{|X_1|}(u)\,u^{-1/2}\,\ell(du)<\infty.
\end{equation}
Then
$$
    \sqrt{n}({\cal R}_\alpha(\widehat F_n)-{\cal R}_\alpha(F))\,\leadsto\,Z_F\qquad\mbox{in $(\R,{\cal B}(\R))$}
$$
for $Z_F\sim{\rm N}_{0,s^2}$ with
\begin{equation}\label{main theorem coroll - eq - 40}
    s^2=s^2_{\alpha,F}:= \int_{\R^2}f_{\alpha,F}(t_0)\,C_F(t_0,t_1)f_{\alpha,F}(t_1)\,(\ell\otimes\ell)(d(t_0,t_1)),
\end{equation}
where
\begin{eqnarray}
    f_{\alpha,F}(t) & := & \frac{1}{(1-2\alpha)F({\cal R}_\alpha(F))+\alpha}\Big((1-\alpha)\eins_{(-\infty,{\cal R}_\alpha(F)]}(t)+\alpha\eins_{({\cal R}_\alpha(F),\infty)}(t)\Big),\quad    \label{main theorem coroll - eq - 50}\\
    C_{F}(t_0,t_1) & := & F(t_0\wedge t_1)(1-F(t_0\vee t_1))+ \sum_{i=0}^1\sum_{k=2}^{\infty}\covi(\eins_{\{X_1 \le t_i\}},\eins_{\{X_k \le t_{i-1}\}}).
    \label{main theorem coroll - eq - 60}
\end{eqnarray}
\end{theorem}

\begin{proof}
Theorem \ref{main theorem} shows that ${\cal R}_\alpha$ is quasi-Hadamard differentiable at $F$ tangentially to $\bL_1\langle\bL_1\rangle$ (w.r.t.\ the norm $\|\cdot\|_{1,\ell}$) with quasi-Hadamard derivative $\dot{\cal R}_{\alpha;F}$ given by (\ref{def of limit of emp process}). The functional delta-method in the form of Theorem \ref{delta method for the bootstrap}(i) and Theorem \ref{CLT by Dede} then imply that $\sqrt{n}({\cal R}_\alpha(\widehat F_n)-{\cal R}_\alpha(F))$ converges in distribution to $\dot{\cal R}_{\alpha;F}(B_F)$, where $B_F$ is an $\bL_1$-valued centered Gaussian random variable with covariance operator $\Phi_{B_F}$ given by (\ref{CLT by Dede - eq}). Now, $\dot{\cal R}_{\alpha;F}(B_F)=-\int f_{\alpha,F}(x)\,B_{F}(x)\,\ell(dx)$ for the $\bL_\infty$-function $f_{\alpha,F}$ given by (\ref{main theorem coroll - eq - 50}). Since $B_{F}$ is a centered Gaussian random element of $\bL_{1}$, and since $f_{\alpha,F}$ represents a continuous linear functional on $\bL_{1}$, the random variable $\dot{\cal R}_{\alpha;F}(B_F)$ is normally distributed with zero mean and variance $\vari[\dot{\cal R}_{\alpha;F}(B_F)]=\ex[\dot{\cal R}_{\alpha;F}(B_F)^2]=\Phi_{B_F}(f_{\alpha,F},f_{\alpha,F})$, and the latter expression is equal to the right-hand side in (\ref{main theorem coroll - eq - 40}).
\end{proof}

Note that when $X_1,X_2,\ldots$ are i.i.d.\ random variables, then (\ref{main theorem coroll - eq - 10}) and (\ref{main theorem coroll - eq - 20}) are clearly satisfied and the expression for the variance $s^2$ in (\ref{main theorem coroll - eq - 40}) simplifies insofar as the sum $\sum_{i=0}^1\sum_{k=2}^\infty(\cdots)$ in (\ref{main theorem coroll - eq - 60}) vanishes, so that $s^2=\ex[U_\alpha(X_1-{\cal R}_\alpha(F))^2]/d_F(\alpha)^2$ with $d_F(\alpha):=(1-2\alpha)F({\cal R}_\alpha(F))+\alpha$. The latter may be seen by applying Hoeffding's variance formula (cf., e.g., Lemma 5.24 in \cite{McNeiletal2005}) to calculate $\vari[(X_1-{\cal R}_{\alpha}(F))^{+}]$ and $\vari[(X_{1}-{\cal R}_{\alpha}(F))^{-}]$ (take into account that by (\ref{def expectiles-based shortfall risk measure}) we have $\ex[U_\alpha(X_1-{\cal R}_\alpha(F))^2] = \vari[U_\alpha(X_1-{\cal R}_\alpha(F))]$, and obviously $\covi((X_1-{\cal R}_{\alpha}(F))^{+},(X_1-{\cal R}_{\alpha}(F))^{-})=0$). But even in this case the variance $s^2$ depends on the unknown distribution function $F$ in a fairly complex way. So, for the derivation of asymptotic confidence intervals the bootstrap results of Section \ref{section bootstrapping the asymptotic distribution} are expected to lead to a more efficient method than the method that is based on the nonparametric estimation of $s^2=s_{\alpha,F}^2$.


\begin{remarknorm}
In the i.i.d.\ case Theorem \ref{main theorem coroll} can also be obtained from classical results on Z-estimators as, for example, Theorem A in Section 7.2.2 of \cite{Serfling1980}. Recently Holzmann and Klar \cite[Theorem 7]{HolzmannKlar2015} showed that, still in the i.i.d.\ case, continuity of $F$ at ${\cal R}_\alpha(F)$ is even necessary in order to obtain a normal limit. It is also worth mentioning that the integrability condition on $F$ in Theorem \ref{main theorem coroll} is slightly stronger than needed, at least in the i.i.d.\ case. Holzmann and Klar \cite[Corollary 4]{HolzmannKlar2015} only assumed that $F$ possesses a finite second absolute moment which is slightly weaker than assuming our integrability condition. 
{\hspace*{\fill}$\Diamond$\par\bigskip}
\end{remarknorm}

\begin{remarknorm}\label{example for cond of main theorem coroll}
The following assertions illustrate the assumptions of Theorem \ref{main theorem coroll}.
\begin{itemize}
    \item[(i)] The integrability condition $\int\sqrt{F(1-F)}\,d\ell<\infty$ holds if $\int\phi^2\,dF<\infty$ for some continuous function $\phi:\R\rightarrow[1,\infty)$ satisfying $\int1/\phi\,d\ell<\infty$ and being strictly decreasing and strictly increasing on $\R_-$ and $\R_+$, respectively.
    \item[(ii)] Condition  (\ref{main theorem coroll - eq - 10}) holds if $\widetilde\phi(n)={\cal O}(n^{-{b}})$ for some $b>1$.
    \item[(iii)] Condition (\ref{main theorem coroll - eq - 20}) implies condition (\ref{main theorem coroll - eq - 10}) with $\widetilde\phi(n)$ replaced by $\widetilde\alpha(n)$.
    \item[(iv)] Condition (\ref{main theorem coroll - eq - 20}) is equivalent to
    $$
        \sum_{n\in\N}n^{-1/2}\int_{(0,\infty)}\widetilde{\alpha}(n)^{1/2}\wedge \pr[|X_1|>x]^{1/2}\,\ell(dx)<\infty.
    $$
    \item[(v)] Condition (\ref{main theorem coroll - eq - 20}) holds if $\int\sqrt{F(1-F)}\,d\ell<\infty$ and $\widetilde\alpha(n)={\cal O}(n^{-b})$ for some $b>1$.
\end{itemize}
See Section \ref{Proof of example for cond of main theorem coroll} for the proofs of these assertions.
{\hspace*{\fill}$\Diamond$\par\bigskip}
\end{remarknorm}


\subsection{Bootstrap consistency}\label{section bootstrapping the asymptotic distribution}

In this section we present two results on bootstrap consistency in the setting of Theorem \ref{main theorem coroll}. 
In the following Theorem \ref{main theorem coroll - bootstrap} we will assume that the random variables $X_1,X_2,\ldots$ are i.i.d. In Theorem \ref{bootstrap results of Radulovic} ahead we will assume that the sequence $(X_i)$ is $\beta$-mixing. We will use $\varrho_{\scriptsize{\rm BL}}$ to denote the bounded Lipschitz metric on the set of all Borel probability measures on $\R$; see the Appendix \ref{appendix QHD and FDM} for the definition of the bounded Lipschitz metric. By $\pr_\xi'$ we will mean the law of a random variable $\xi$ under $\pr'$, and as before ${\rm N}_{0,s^2}$ refers to the centered normal distribution with variance $s^2$.

\begin{theorem}\label{main theorem coroll - bootstrap}
Let $(X_i)$ be a sequence of i.i.d.\ real-valued random variables on some probability space $(\Omega,{\cal F},\pr)$. Denote by $F$ the distribution function of the $X_i$, and assume that $F$ is continuous at ${\cal R}_\alpha(F)$ and that $\int\phi^2dF<\infty$ for some continuous function $\phi:\R\rightarrow[1,\infty)$ satisfying $\int1/\phi\,d\ell<\infty$ (in particular $F\in\F_1$). Let $\widehat F_n$ be the empirical distribution function of $X_1,\ldots,X_n$ as defined in (\ref{def edf}). Let $(W_{ni})$ be a triangular array of nonnegative real-valued random variables on another probability space $(\Omega',{\cal F}',\pr')$. Set $(\overline\Omega,\overline{\cal F},\overline\pr):=(\Omega\times\Omega',{\cal F}\otimes{\cal F}',\pr\otimes\pr')$ and define the map $\widehat F_n^*:\overline\Omega\rightarrow\F_1$ by
\begin{equation}\label{bootstrap results of VdV-W - 10}
    \widehat F_n^*(\omega,\omega')\,:=\,\frac{1}{n}\sum_{i=1}^n W_{ni} (\omega')\,\eins_{[X_i(\omega),\infty)}.
\end{equation}
Finally assume that one of the following two settings is met:
\begin{itemize}
    \item[(a)] {\em (Efron's bootstrap)} The random vector $(W_{n1},\ldots,W_{nn})$ is multinomially distributed according to the parameters $n$ and $p_1=\cdots=p_n=\frac{1}{n}$ for every $n\in\N$.
    \item[(b)] {\em (Bayesian bootstrap)} $W_{ni}=Y_i/\overline{Y}_n$ for every $n\in\N$ and $i=1,\ldots,n$, where $\overline{Y}_n:=\frac{1}{n}\sum_{j=1}^nY_j$ and $(Y_j)$ is any sequence of nonnegative i.i.d.\ random variables on $(\Omega',{\cal F}',\pr')$ with distribution $\mu$ which satisfies $\int_0^\infty{\mu}[(x,\infty)]^{1/2}\,dx<\infty$ and whose standard deviation coincides with its mean and is strictly positive.
\end{itemize}
Then
\begin{equation}\label{bootstrap results of VdV-W - 20}
    \lim_{n\to\infty}\pr\big[\big\{\omega\in\Omega:\,\varrho_{\scriptsize{\rm BL}}\big(\pr'_{\sqrt{n}({\cal R}_\alpha(\widehat F_n^{*}(\omega,\cdot))-{\cal R}_\alpha(\widehat F_n(\omega)))},{\rm N}_{0,s^2}\big)\ge\delta\big\}\big]=\,0\quad\mbox{ for all }\delta>0,
\end{equation}
where $s^2=s_{\alpha,F}^2$ is given by (\ref{main theorem coroll - eq - 40}) (with $C_F(t_0,t_1)=F(t_0\wedge t_1)(1-F(t_0\vee t_1))$).
\end{theorem}

\begin{proof}
First of all note that ${\cal R}_\alpha(\widehat F_n^*)$ may be verified to be $({\cal F},{\cal B}(\R))$-measurable in a similar way like ${\cal R}_\alpha(\widehat F_n)$. Theorem \ref{main theorem} shows that ${\cal R}_\alpha$ is quasi-Hadamard differentiable at $F$ tangentially to $\bL_1\langle\bL_1\rangle$ (w.r.t.\ the norm $\|\cdot\|_{1,\ell}$) with linear quasi-Hadamard derivative $\dot{\cal R}_{\alpha;F}$ given by (\ref{def of limit of emp process}). The functional delta-method in the form of Theorem \ref{delta method for the bootstrap}(ii) along with Theorems \ref{main theorem coroll} and \ref{bootstrap results of VdV-W} then implies that (\ref{bootstrap results of VdV-W - 20}) with ${\rm N}_{0,s^2}$ replaced by the law of $\dot{\cal R}_{\alpha;F}(B_F)$ holds, where $B_F$ is an $\bL_1$-valued centered Gaussian random variable with covariance operator $\Phi_{B_F}$ given by (\ref{CLT by Dede - eq}) (with $C_F(t_0,t_1)=F(t_0\wedge t_1)(1-F(t_0\vee t_1))$). As in the proof of Theorem \ref{main theorem coroll} we obtain $\dot{\cal R}_{\alpha;F}(B_F)\sim{\rm N}_{0,s^2}$. For the application of Theorem \ref{main theorem coroll} note that $\int\sqrt{F(1-F)}\,d\ell<\infty$ is ensured by the assumption $\int\phi^2dF<\infty$; cf.\ Remark \ref{example for cond of main theorem coroll}(i).
\end{proof}

We now turn to the case where the observations $X_1,X_2,\ldots$ may be dependent. We focus on the so-called circular bootstrap \cite{PolitisRomano1992,Radulovic1996}, which is only a slight modification of the moving blocks bootstrap \cite{Buehlmann1994,Kuensch1989,LiuSingh1992,NaikNimbalkarRajarshi1994}. To ensure that in the following $\widehat F_n^*$ is the distribution function of a {\em probability} measure, we assume that $n$ ranges only over $\N_m:=\{m^k:k=1,2,\ldots\}$ for some arbitrarily fixed integer $m\geq 2$. Let $(\ell_n)$ be a sequence in $\N$ such that $\ell_n<n$ is a divisor of $n$ and $\ell_n\nearrow\infty$ as $n\rightarrow\infty$, and set $k_n:=n/\ell_n$. Let $(I_{nj})_{n\in\N,\,1\le j\le k_n}$ be a triangular array of random variables on another probability space $(\Omega',{\cal F}',\pr')$ such that $I_{n1},\ldots,I_{nk_n}$ are i.i.d.\ according to the uniform distribution on $\{1,\ldots,n\}$ for every $n\in\N$. Set $(\overline\Omega,\overline{\cal F},\overline\pr):=(\Omega\times\Omega',{\cal F}\otimes{\cal F}',\pr\otimes\pr')$ and define the map $\widehat F_{n}^*:\overline\Omega\rightarrow\F_1$ by (\ref{bootstrap results of VdV-W - 10})
with
\begin{equation}\label{example for tau - beta mixing - 2}
    W_{ni}(\omega')\, := \, \sum_{j=1}^{k_n}\Big(\eins_{\{I_{nj}\le i\le (I_{nj}+\ell_n-1)\wedge n\}}(\omega')+\eins_{\{I_{nj}+\ell_n-1 > n,\,1 \le i \le I_{nj}+\ell_n-1-n\}}(\omega')\Big).
\end{equation}
At an informal level this means that given a sample $X_1,\ldots,X_{n}$, we pick $k_n$ blocks of length $\ell_n$ in the (artificially) extended sample $X_1,\ldots,X_{n},X_{n+1},\ldots,X_{n}+\ell_n-1$ (with $X_{n+i}:=X_i$, $i=1,\ldots,\ell_n-1$) where the start indices $I_{n1},I_{n2},\ldots,I_{nk_n}$ are chosen independently and uniformly in the set of all indices $\{1,\ldots,n\}$.
The bootstrapped empirical distribution function $\widehat F_{n}^*$ is then defined to be the distribution function of the discrete 
probability measure with atoms $X_1,\ldots,X_{n}$ carrying masses $W_{n1},\ldots,W_{nn}$ respectively, where $W_{ni}$ specifies the number of blocks which contain $X_i$.

\begin{theorem}{\em (Circular bootstrap)}\label{bootstrap results of Radulovic}
Let $(X_i)$ be a strictly stationary sequence of real-valued random variables on some probability space $(\Omega,{\cal F},\pr)$. Denote by $F$ the distribution function of the $X_i$, and assume that $F$ is continuous at ${\cal R}_\alpha(F)$ and that $\int|x|^p\,dF(x)<\infty$ for some $p>2$ (in particular $F\in\F_1$). Assume that $(X_i)$ is $\beta$-mixing with mixing coefficients $\beta(i)={\cal O}(i^{-b})$ for some $b>p/(p-2)$. Let $\widehat F_n$ be the empirical distribution function of $X_1,\ldots,X_n$ as defined in (\ref{def edf}). Let $(\ell_n)$ and $(k_n)$ be as above and assume that $\ell_n={\cal O}(n^{\gamma})$ for some $\gamma\in(0,(p-2)/(2(p-1)))$. Let $(\overline\Omega,\overline{\cal F},\overline\pr)$ be as above and the map $\widehat F_{n}^*:\overline\Omega\rightarrow\F_1$ be given by (\ref{bootstrap results of VdV-W - 10}) and (\ref{example for tau - beta mixing - 2}). Then
$$
    \lim_{n\to\infty}\pr\big[\big\{\omega\in\Omega:\,\varrho_{\scriptsize{\rm BL}}\big(\pr'_{\sqrt{n}({\cal R}_\alpha(\widehat F_{n}^{*}(\omega,\cdot))-{\cal R}_\alpha(\widehat F_{n}(\omega)))},{\rm N}_{0,s^2}\big)\ge\delta\big\}\big]=\,0\quad\mbox{ for all }\delta>0,
$$
where $s^2=s_{\alpha,F}^2$ is given by (\ref{main theorem coroll - eq - 40}).
\end{theorem}

\begin{proof}
One can argue as in the proof of Theorem \ref{main theorem coroll - bootstrap} (with Theorem \ref{bootstrap results of Radulovic EP} in place of Theorem \ref{bootstrap results of VdV-W}). Take into account that $\beta$-mixing implies ergodicity (see Section 2.5 in \cite{Bradley2005}) and that $\int|x|^p\,dF(x)<\infty$ and $\beta(i)={\cal O}(i^{-b})$ (with $p>2$ and $b>p/(p-2)>1$) imply $\int\sqrt{F(1-F)}\,d\ell<\infty$ and (\ref{main theorem coroll - eq - 20}); cf.\ Remark \ref{example for cond of main theorem coroll} (i) and (iv).
\end{proof}

\subsection{Qualitative robustness}\label{section qualitative robustness}

Consider the nonparametric statistical infinite product model
$$
    \big(\Omega,{\cal F},\{\pr^\theta:\theta\in\Theta\}\big):=\big(\R^\N,{\cal B}(\R)^{\otimes\N},\{P_F^{\otimes\N}:F\in\F_1\}\big),
$$
where $P_F$ is the Borel probability measure on $\R$ associated with the distribution function $F$. Let $X_i$ be the $i$-th coordinate projection on $\Omega=\R^\N$, and note that $X_1,X_2,\ldots$ are i.i.d.\ with distribution function $F$ under $\pr^F$ for every $F\in\F_1$. Let $\widehat F_n$ be the empirical distribution function of $X_1,\ldots,X_n$ as defined in (\ref{def edf}) and set $\widehat{\cal R}_n:={\cal R}_\alpha(\widehat F_n)$. We will say that the sequence of estimators $(\widehat{\cal R}_n)$ is qualitatively robust on a given set $\bG\subseteq\F_1$ if for every $F\in\bG$ and $\varepsilon>0$ there exists a $\delta>0$ such that for every $n\in\N$
$$
    G\in\bG,\quad \varrho_{\mbox{\scriptsize{\rm P}}}(P_F,P_G)\le\delta\qquad\Longrightarrow\qquad \varrho_{\mbox{\scriptsize{\rm P}}}(\pr^{F}\circ\widehat{\cal R}_n^{-1},\pr^{G}\circ\widehat{\cal R}_n^{-1})\le\varepsilon,
$$
where $\varrho_{\mbox{\scriptsize{\rm P}}}$ refers to the Prohorov metric on the set of all Borel probability measures on $\R$. Theorem \ref{thm qualitative robustness} ahead shows that the sequence of estimators $(\widehat{\cal R}_n)$ is qualitatively robust on every so-called w-sets in $\F_1$. Following \cite{Kraetschmeretal2015}, we say that a subset $\bG\subseteq\F_1$ is a w-set in $\F_1$ if the relative $1$-weak topology and the relative weak topology coincide on $\bG$. Several characterizations and examples of w-sets in $\F_1$ have been worked out in \cite{Kraetschmeretal2015}. The examples include the class of distribution functions of all normal distributions, the class of distribution functions of all Gamma distributions with location parameter $0$, the class of distributions functions of all Pareto distributions on $[\overline{c},\infty)$ with shape parameter $a\in[\alpha,\infty)$ for any fixed $\alpha>1$ and $\overline{c}>0$, and others.

The following theorem is an immediate consequence of Theorem 3.8 and Lemma 3.4 in \cite{Zaehle2016} and Theorems \ref{1 weak continuity of R alpha} and \ref{Consistency thm}.

\begin{theorem}\label{thm qualitative robustness}
The sequence of estimators $(\widehat{\cal R}_n)$ is qualitatively robust on every w-set in $\F_1$.
\end{theorem}


\subsection{Comparison to other empirical risk measures}

As already mentioned in the introduction, the expectile-based risk measure $\rho_{\alpha}$ has recently attracted some attention as a tool of quantitative risk management. Already established alternatives are the Value at Risk ${\rm V@R}_{\alpha}(X):=F_{X}^{\leftarrow}(\alpha)$ and the Average Value at Risk ${\rm AV@R}_{\alpha}(X):=\frac{1}{1-\alpha}\int_{(\alpha,1)}F_{X}^{\leftarrow}(s)\,\ell(ds)$ at level $\alpha\in (0,1)$,
where $F_{X}^{\leftarrow}$ denotes the left-continuous quantile function of the distribution function $F_{X}$ of $X$. Whereas ${\rm V@R}_{\alpha}$ may be evaluated at any random variable on the underlying probability space, ${\rm AV@R}_{\alpha}$ is restricted to $L^1$-random variables just as $\rho_\alpha$. Since they both are law-invariant, they may be associated with statistical functionals in the same way as expectile-based risk measures. The corresponding nonparametric empirical estimators can be obtained by evaluating these functionals at the empirical distribution function as defined in (\ref{def edf}). In Table \ref{Table 1} below these estimators are compared to the empirical expectile-based risk measure. To keep the discussion tight, we shall restrict considerations to i.i.d.\ samples.

Let $\F_0$ be the set of all distribution functions on $\R$, $\F_2$ be the subset of all $F\in\F_0$ satisfying $\int x^2\,dF(x)<\infty$, and $\F_0^\alpha$ be the subset of all $F\in\F_0$ having a unique $\alpha$-quantile. Moreover let $F_2^{\alpha;{\rm{\scriptsize{e}}}}$ be the sets of all $F\in\F_2$ being continuous at ${\cal R}_\alpha(F)$, and $F_2^{\alpha;{\rm{\scriptsize{a}}}}$ be the analogue for the Average Value at Risk at level $\alpha$. The first two columns of the first two lines of Table \ref{Table 1} can be derived by means of the classical theory of L-statistics as presented in \cite{Serfling1980,van der Vaart 1998}. It is moreover known from \cite{FeldmanTucker1966} that strong consistency of the empirical $\alpha$-quantile ($\vatr_\alpha$) does not hold for $F\in\F_0\setminus\F_0^\alpha$. The third column of the first two lines is known from the results of \cite{HolzmannKlar2015}; see also our elaborations above. The results of \cite{HolzmannKlar2015} moreover show that asymptotic normality of the empirical $\alpha$-expectile cannot be obtained for $F\in\F_2\setminus\F_2^{\alpha;{\rm{\scriptsize{e}}}}$. The first and the second column of the third line are known from the discussion at the end of Section 2 in \cite{Kraetschmeretal2015} and Section 4.3 in \cite{Kraetschmeretal2015}, respectively. The third column of the third line is justified by Theorem \ref{thm qualitative robustness} above. It follows by Hampel's theorem that robustness of the empirical $\alpha$-quantile ($\vatr_\alpha$) cannot be obtained on sets larger than $\F_0^\alpha$. On the other hand, it is not clear to us whether or not robustness of the empirical Average Value at Risk and the empirical expectile can be obtained on sets that are not w-sets in $\F_1$.

\begin{table}[hbt]\centering
\begin{tabular}{lccc}
    & $\vatr_\alpha$ & $\avatr_\alpha$ & $\rho_\alpha$ \\
    \hline
    strong consistency & for $F\in\F_0^\alpha$ & for $F\in\F_1$ & for $F\in\F_1$\\
    asymptotic normality & for $F\in\F_0^\alpha$ & for $F\in\F_2^{\alpha;{\rm{\scriptsize{a}}}}$ & for $F\in\F_2^{\alpha;{\rm{\scriptsize{e}}}}$\\
    qualitative robustness & on $\F_0^\alpha$ & on w-sets in $\F_1$ & on w-sets in $\F_1$\\
\end{tabular}
\caption{Comparison of empirical estimators of $\vatr_\alpha$, $\avatr_\alpha$, and $\rho_\alpha$.}
\label{Table 1}
\end{table}


\section{Parametric estimation of ${\cal R}_\alpha(F)$}\label{ParamEstim}

In this section we consider a parametric statistical model $(\Omega,{\cal F},\{\pr^\theta:\theta\in\Theta\})$, where the parameter set $\Theta$ is any topological space. In Section \ref{section asymptotic distribution - param} we will also impose some additional structure on $\Theta$. For every $n\in\N$ we let $\widehat\theta_n:\Omega\to\Theta$ be any map, which should be seen as an estimator for $\theta$. For every $\theta\in\Theta$ we fix a distribution function $F_\theta\in\F_1$, which can be seen as a characteristic derived from the parameter $\theta$. In particular, $\widehat F_n:=F_{\widehat\theta_n}$ can be seen as an estimator for $F_\theta$.


\subsection{Strong consistency}\label{section strong consistency - param}

Here we need no further assumptions on the topological space $\Theta$. The following Theorem \ref{Consistency thm - parametric} is an immediate consequence of Theorem \ref{1 weak continuity of R alpha}.

\begin{theorem}\label{Consistency thm - parametric}
Let $\theta_0\in\Theta$ and assume that the mapping $\theta\mapsto F_\theta$ is $1$-weakly sequentially continuous at $\theta_0$. Moreover assume that $\widehat\theta_n\to\theta_0$ $\pr^{\theta_0}$-a.s. Then, under $\pr^{\theta_0}$, the estimator ${\cal R}_\alpha(F_{\widehat\theta_n})$ is strongly consistent for ${\cal R}_\alpha(F_{\theta_0})$ in the sense that
$$
    {\cal R}_\alpha(F_{\widehat\theta_n})\to{\cal R}_\alpha(F_{\theta_0})\qquad\pr^{\theta_0}\mbox{-a.s.}
$$
\end{theorem}

Note that in Theorem \ref{Consistency thm  - parametric} 
the concept of strong consistency is used as a purely analytical property of the sequence $({\cal R}_\alpha(\widehat F_n))$ without further measurability condition on 
${\cal R}_\alpha(\widehat F_n)$. Example \ref{Example Consitency Parametric - Lognormal} will illustrate the conditions of Theorem  \ref{Consistency thm - parametric}.

\begin{remarknorm}\label{remark on Consistency thm - parametric}
Recall from Lemma 3.4 in \cite{Kraetschmeretal2012} that a sequence $(F_n)\subseteq\F_1$ converges $1$-weakly to some $F_0\in\F_1$ if and only if $\int f\,dF_n\to\int f\,dF_0$ for all $f\in{\cal C}_1$. Thus the mapping $\theta\mapsto F_\theta$ is $1$-weakly sequentially continuous at $\theta_0$ if and only if for every sequence $(\theta_n)\subseteq\Theta$ with $\theta_n\rightarrow\theta_0$ we have $\int f\,dF_{\theta_n}\rightarrow\int f\,dF_{\theta_0}$ for all $f\in{\cal C}_1$.
{\hspace*{\fill}$\Diamond$\par\bigskip}
\end{remarknorm}

%

\begin{examplenorm}\label{Example Consitency Parametric - Lognormal}
Let $\Theta:=\R\times (0,\infty)$ and $F_{(m,s^2)}$ be the distribution function of the log-normal distribution ${\rm LN}_{(m,s^2)}$ with parameters $(m,s^2)\in\Theta$. Recall that ${\rm LN}_{(m,s^2)}$ possesses the Lebesgue density
$$
    f_{(m,s^2)}(x)
    :=
    \left\{
    \begin{array}{lll}
        (2\pi s^2)^{-1/2}\,x^{-1}\,e^{-\{\log(x)-m\}^2/\{2s^2\}} & , & x> 0\\
        0 & , & x\le 0
    \end{array}
    \right..
$$
It is shown in Section \ref{Proof of Example Consitency Parametric - Lognormal} ahead that the mapping $(m,s^2)\mapsto F_{(m,s^2)}$ is $1$-weakly sequentially continuous at every $(m_0,s_0^2)\in\Theta$. Further, in the corresponding infinite statistical product model $(\R^\N,{\cal B}(\R)^{\otimes\N},\{{\rm LN}_{(m,s^2)}:(m,s^2)\in\Theta\})$ a maximum likelihood estimator $(\widehat m_n,\widehat s_n^2)$ for $(m,s^2)$ is given by \begin{eqnarray}\label{Example Consitency Parametric - Lognormal - EQ}
    \widehat m_n(x_1\,x_2,\ldots)
    & := &
    \left\{
    \begin{array}{lll}
        \frac{1}{n}\sum_{i=1}^n\log(x_i) & , & \min_{i = 1,\dots,n}x_{i} > 0\\
        \overline{m} & , & \min_{i = 1,\dots,n}x_{i}\le 0
    \end{array}
    \right.,\\
    \widehat s_n^2(x_1\,x_2,\ldots)
    & := &
        \left\{
    \begin{array}{lll}
        \frac{1}{n}\sum_{i=1}^n\big(\log(x_i)-\widehat m_n(x_1\,x_2,\ldots)\big)^2 & , & \min_{i = 1,\dots,n}x_{i} > 0\\
        \overline{s}^2 & , & \min_{i = 1,\dots,n}x_{i} \le 0
    \end{array}
    \right.\nonumber
\end{eqnarray}
for any fixed $\overline{m}\in\R$ and $\overline{s}^2>0$. By using the classical strong law of large numbers, $(\widehat m_n,\widehat s_n^2)$ is easily shown to be strongly consistent.
{\hspace*{\fill}$\Diamond$\par\bigskip}
\end{examplenorm}

\begin{examplenorm}\label{Example Consitency Parametric - Pareto}
Let $\Theta:=(1,\infty)$ and $F_{a}$ be the distribution function of the Pareto distribution ${\rm Par}_{a,\overline{c}}$ with unknown tail-index $a > 0$ and {\em known} location parameter $\overline{c} > 0$. Recall that ${\rm Par}_{a,\overline{c}}$ possesses the Lebesgue density
$$
    f_{a}(x):=
    \left\{
    \begin{array}{lll}
        a\,\overline{c}^{a}\,x^{-(a+1)} & , & x>\overline{c}\\
        0 & , & x\le\overline{c}
    \end{array}
    \right..
$$
It may be verified very easily that the mapping $a\mapsto F_{a}$ is $1$-weakly sequentially continuous at every $a_0\in\Theta$. Further, in the corresponding infinite statistical product model $(\R^\N,{\cal B}(\R)^{\otimes\N},\{{\rm Par}_{a,_{a,\overline{c}}}: a\in\Theta\})$ a maximum likelihood estimator $\widehat{a}_n$ for the tail-index $a$ is given by
\begin{equation}\label{ML Pareto}
    \widehat{a}_n(x_{1},x_{2},\dots) :=
    \left\{
    \begin{array}{lll}
        \big(\frac{1}{n}\sum_{i=1}^{n} (\log(x_{i}) - \log(\overline{c})\big)^{-1} & , & \min_{i=1,\dots, n} x_{i} >\overline{c}\\
        \overline{a} & , & \min_{i=1,\dots, n} x_{i}\le\overline{c}
    \end{array}
    \right.
\end{equation}
for any fixed $\overline{a} > 1$. Since the expectation of the logarithm of a ${\rm Par}_{a,\overline{c}}$-distributed random variable equals $\log(\overline{c})+1/a$, it follows easily by the classical strong law of large numbers that this estimator is strongly consistent.

A popular alternative estimator for the tail-index $a$ is the so-called Hill estimator
\begin{equation}\label{Hill Pareto}
    \widehat{a}^{H}_{n,k_{n}}(x_{1},x_{2},\dots) :=
    \left\{
    \begin{array}{lll}
        \big\{\frac{1}{k_{n}}\sum_{i=1}^{k_n} \log(x_{n:n - i + 1}) - \log(x_{n:n - k_{n}})\big\}^{-1} & , & \min_{i=1,\dots, k_{n}} x_{i} >\overline{c}\\
        \overline{a} & , & \min_{i=1,\dots, k_{n}} x_{i} \le\overline{c}
    \end{array}
    \right.
\end{equation}
for any fixed $\overline{a} > 1$, where $n\ge 2$, $k_{n}\in\{1,\dots,n - 1\}$, and $x_{n:1}\leq\dots\leq x_{n:n}$ denotes an increasing ordering of $x_{1},\dots,x_{n}$. It is known from \cite{DeheuvelsHaeuslerMason1988} that $(\widehat{a}^{H}_{n,k_{n}})$ is strongly consistent for $a$ whenever $k_{n}/\log(\log(n))\to 0$ and $k_{n}/n\to 0$ as $n\to\infty$.
{\hspace*{\fill}$\Diamond$\par\bigskip}
\end{examplenorm}


\subsection{Asymptotic distribution}\label{section asymptotic distribution - param}

In this section we assume that $\Theta$ and $\Upsilon_{0}$ are subsets of a vector space $\Upsilon$ equipped with a separable norm $\|\cdot\|_{\Upsilon}$. We denote by ${\cal B}(\Upsilon)$ the Borel $\sigma$-algebra on $(\Upsilon,\|\cdot\|_{\Upsilon})$ and by ${\cal B}_1$ the Borel $\sigma$-algebra on $(\bL_1,\|\cdot\|_{1,\ell})$. Moreover we define a map $\mathfrak{F}:\Theta\rightarrow\F_1 (\subseteq\bL_0)$ by
$$
    \mathfrak{F}(\theta):=F_\theta.
$$
In the following theorem we assume that $\mathfrak{F}$ is Hadamard differentiable at some $\theta_0\in\Theta$ tangentially to $\Upsilon_{0}$ with trace $\bL_1$ (in the sense of Definition \ref{definition quasi hadamard} and Remark \ref{definition quasi hadamard - remark}(iii). This means that there exists a continuous map $\dot \mathfrak{F}_{\theta_0}:\Upsilon_{0}\rightarrow\bL_1$ (the Hadamard derivative) such that
$$
    \lim_{n\to\infty}\Big\|\dot \mathfrak{F}_{\theta_0}(\tau)-\frac{F_{\theta_0+\varepsilon_n\tau_n}- F_{\theta_0}}{\varepsilon_n}\Big\|_{1,\ell}\,=\,0
$$
holds for each triplet $(\tau,(\tau_n),(\varepsilon_n))$ with $\tau\in\Upsilon_{0}$, $(\tau_n)\subseteq\Upsilon$ satisfying $(\theta_0 + \varepsilon_{n}\tau_{n})\subseteq\Theta$ as well as $\|\tau_n-\tau\|_{\Upsilon}\to 0$, and $(\varepsilon_n)\subseteq(0,\infty)$ satisfying $\varepsilon_n\to 0$. Recall that $F_1-F_2\in\bL_{1}$ holds for every $F_1,F_2\in\F_{1}$.


\begin{theorem}\label{Asymptotc Distribution thm - parametric}
Let $\theta_0\in\Theta$ and $(a_n)$ be a sequence of positive real numbers tending to $\infty$. Let $\widehat\theta_n:\Omega\to\Theta$ be any map such that $a_n(\widehat\theta_n-\theta_0)$ is $({\cal F},{\cal B}(\Upsilon))$-measurable and $a_n(\widehat\theta_n-\theta_0)\leadsto Y_{\theta_0}$ under $\pr^{\theta_0}$ for some $(\Upsilon,{\cal B}(\Upsilon))$-valued random variable $Y_{\theta_0}$ taking values only in $\Upsilon_0$. Assume that $a_n(F_{\widehat\theta_n}-F_{\theta_0})$ is $({\cal F},{\cal B}_1)$-measurable. Further assume that the map $\mathfrak{F}:\Theta\rightarrow\F_1(\subseteq\bL_0)$ is Hadamard differentiable at $\theta_0$ tangentially to $\Upsilon_0$ with trace $\bL_1$, and consider the Hadamard derivative $\dot \mathfrak{F}_{\theta_0}:\Upsilon_{0}\rightarrow\bL_1$. If $F_{\theta_{0}}$ is continuous at $\cR_{\alpha}(F_{\theta_{0}})$, then
\begin{equation}\label{Asymptotc Distribution thm - parametric - eq}
    a_n\big({\cal R}_\alpha(F_{\widehat\theta_n})-{\cal R}_\alpha(F_{\theta_0})\big)\,\leadsto\,\dot{\cal R}_{\alpha;{F_{\theta_0}}}(\dot \mathfrak{F}_{\theta_0}(Y_{\theta_0}))\qquad\mbox{in $(\R,{\cal B}(\R))$}
\end{equation}
under $\pr^{\theta_0}$, where $\dot{\cal R}_{\alpha;F}$ is defined by (\ref{def of limit of emp process}).
\end{theorem}

\begin{proof}
In view of $a_n(\widehat\theta_n-\theta_0)\leadsto Y_{\theta_0}$ under $\pr^{\theta_0}$ and the Hadamard differentiability of $\mathfrak{F}$ tangentially to $\Upsilon_0$ with trace $\bL_1$, the functional delta-method in the form of Theorem \ref{delta method for the bootstrap} yields $a_n(F_{\widehat\theta_n}-F_{\theta_0})\leadsto\dot{\mathfrak{F}}_{\theta_0}(Y_{\theta_0})$ in $(\bL_1,{\cal B}_1,\|\cdot\|_{1,\ell})$ under $\pr^{\theta_0}$. Then Theorem \ref{main theorem} and another application of the functional delta-method in the form of Theorem \ref{delta method for the bootstrap} give (\ref{Asymptotc Distribution thm - parametric - eq}).
\end{proof}

If $\Theta$ is an open subset of a Euclidean space, then we may find a convenient criterion to guarantee the condition of differentiability required for $\mathfrak{F}$ in Theorem \ref{Asymptotc Distribution thm - parametric}. The following lemma provides a criterion in terms of conditions on the map $\mathfrak{f}:\Theta\times\R\rightarrow[0,1]$ defined by
$$
    \mathfrak{f}(\theta,x):= F_{\theta}(x).
$$
In this lemma $\|\cdot\|$ and $\langle \cdot , \cdot \rangle$ stand for the Euclidean norm and Euclidean scalar product on $\R^{d}$, and ${\rm grad}_{\theta}\,\mathfrak{f}(\theta,x)$ denotes the gradient of the function $\mathfrak{f}(\,\cdot\,,x)$ at $\theta$ for any fixed $x$.

\begin{lemma}\label{Asymptotc Distribution cor - parametric}
Let $\Theta\subseteq\R^d$ be open and $\theta_{0}\in\Theta$. Let ${\cal V}$ denote some open neighbourhood of $\theta_{0}$ in $\Theta$ such that for every $x$ the map $\mathfrak{f}(\,\cdot\,,x)$ is continuously differentiable on ${\cal V}$. Furthermore, let $\mathfrak{h}:\R\rightarrow\R$ be an $\ell$-integrable function such that
$$
    \sup_{\theta\in{\cal V}}\,\|{\rm grad}_{\theta}\,\mathfrak{f}(\theta,x)\|\leq \mathfrak{h}(x)\qquad\ell\mbox{-a.e.\ }x.
$$
Then $\mathfrak{F}:\Theta\rightarrow\F_1(\subseteq\bL_0)$ is Hadamard differentiable at $\theta_0$ (tangentially to the whole space $\R^d$) with trace $\bL_1$, and the Hadamard derivative $\dot \mathfrak{F}_{\theta_0}:\R^{d}\rightarrow\bL_1$ is given by
$$
    \dot \mathfrak{F}_{\theta_0}(\tau)(\,\cdot\,):=\langle{\rm grad}_{\theta}\,\mathfrak{f}(\theta_0,\,\cdot\,),\tau\rangle,\qquad \tau\in\Theta.
$$
\end{lemma}

The proof of Lemma \ref{Asymptotc Distribution cor - parametric} may be found in Section \ref{vereinfachtes Kriterium}.

\begin{examplenorm}\label{Example Asymptotic Distribution Parametric - Lognormal}
Consider the subset $\Theta:=\R\times (0,\infty)$ of $\Upsilon:=\R^2$, and let $F_{(m,s^2)}$ be the distribution function of the log-normal distribution ${\rm LN}_{(m,s^2)}$ with parameters $(m,s^2)\in\Theta$ as in Example \ref{Example Consitency Parametric - Lognormal}. Moreover consider the map $\mathfrak{F}:\Theta\rightarrow\F_1(\subseteq\bL_0)$ defined by $\mathfrak{F}(m,s^2):=F_{(m,s^2)}$. It is shown in Section \ref{Proof of Example Asymptotic Distribution Parametric - Lognormal} ahead (using Lemma \ref{Asymptotc Distribution cor - parametric}) that for every fixed $(m_0,s_0^2)\in\Theta$ the map $\mathfrak{F}$ is Hadamard differentiable at $(m_0,s_0^2)\in\Theta$ (tangentially to the whole space $\Upsilon=\R^2$) with trace $\bL_1$ and Hadamard derivative $\dot \mathfrak{F}_{(m_0,s_0^2)}:\R^{2}\rightarrow\bL_1$ given by
\begin{equation}\label{Example Asymptotic Distribution Parametric - Lognormal - derivative}
    \dot\mathfrak{F}_{(m_{0},s_{0}^{2})}(\tau_1,\tau_2)(x):=
    \left\{
    \begin{array}{lll}
          -\Big(\frac{\tau_1}{s_{0}} + \frac{(\log(x) - m_{0})\,\tau_2}{s_{0}^{3}}\Big)\phi_{(0,1)}\Big(\frac{\log(x) - m_{0}}{s_{0}}\Big) & , & x> 0\\
        0 & , & x\le 0
    \end{array}
    \right.,
\end{equation}
where $\phi_{(0,1)}$ is the standard Lebesgue density of the standard normal distribution.

Further, it may be verified easily that the family $\{{\rm LN}_{(m,s^{2})}:(m,s^{2})\in\Theta\}$ satisfies the assumptions of Theorem 6.5.1 in \cite{LehmannCasella1998}. Therefore the maximum likelihood estimator $(\widehat m_n,s_n^2)$ given by (\ref{Example Consitency Parametric - Lognormal - EQ}) in the corresponding infinite statistical product model satisfies
$$
    \sqrt{n}
    \left(
    \left[
    \begin{array}{l}
        \widehat m_n\\
        s_n^2
    \end{array}
    \right]
    -
    \left[
    \begin{array}{l}
        m_0\\
        s_0^2
    \end{array}
    \right]
    \right)\,\leadsto\,Y_{(m_0,s_0^2)}
    \qquad\mbox{in $(\Upsilon,{\cal B}(\Upsilon))$}
$$
under $\pr^{(m_0,s_0^2)}$ for every $(m_0,s_0^2)\in\Theta$, where $Y_{(m_0,s_0^2)}$ is bivariate centered normal with covariance matrix ${\cal I}(m_{0},s_{0}^{2})^{-1}$. Here ${\cal I}(m_{0},s_{0}^{2})$ denotes the Fisher information matrix at $(m_{0},s_{0}^{2})$, and  elementary calculations show that
$$
    {\cal I}(m_{0},s_{0}^{2})^{-1} =
    \left[
    \begin{array}{ll}
        s_{0}^{2} & 0\\
        0 & 2 (s_{0}^{2})^{2}
    \end{array}
    \right].
$$

Now Theorem \ref{Asymptotc Distribution thm - parametric} shows that $\sqrt{n}({\cal R}_\alpha(F_{(\widehat m_n,\widehat s_n^2)})-{\cal R}_\alpha(F_{(m_0,s_0^2)}))$ converges in distribution to $Z:=\dot\cR_{\alpha;F_{(m_{0},s_{0}^{2})}}(\dot\mathfrak{F}_{(m_{0},s_{0}^{2})}(Y_{(m_{0},s_{0}^{2}}))$. In Section \ref{Proof of Example Asymptotic Distribution Parametric - Lognormal} it is shown that the limit $Z$ is centered normal with variance
\begin{eqnarray}\label{Example Asymptotic Distribution Parametric - Lognormal - variance}
    \lefteqn{e^{2m_0 + s_0^{2}}\, \Big(1 + 2\Big\{s_{0} + \frac{\phi_{(0,1)}\big(\varphi(m_{0}, s_{0}^{2})\big)}{1 - \alpha - (1 - 2\alpha)\Phi_{(0,1)}\big(\varphi(m_{0}, s_{0}^{2})\big)}\Big\}^{2}\,\Big)}\\
    & & \qquad\qquad\qquad\qquad\qquad\times\,\Big(\frac{(1 - \alpha) - (1 - 2\alpha)\Phi_{(0,1)}\big(\varphi(m_0,s_0^{2})\big)}{(1-2\alpha)F_{(m_{0},s_{0}^{2})}({\cal R}_\alpha(F_{(m_{0},s_{0}^{2})}))+\alpha}\Big)^{2},\nonumber
\end{eqnarray}
where $\varphi(m_{0},s_{0}^{2}) := (m_{0} + s_{0}^{2} - \log(\cR_{\alpha}(F_{(m_{0},s_{0}^{2}})))/2$ (note that we may show $\cR_{\alpha}(F_{(m_{0},s_{0}^{2})}) > 0$; cf.\ Section \ref{Proof of Example Asymptotic Distribution Parametric - Lognormal} below) and $\Phi_{(0,1)}$ denotes the distribution function of the standard normal distribution. Note that $1 - \alpha - (1 - 2\alpha)~\Phi_{(0,1)}(z) = (1-\alpha)\Phi_{(0,1)}(-z) + \alpha \Phi_{(0,1)}(z) > 0$ holds for every $z\in\R$.
{\hspace*{\fill}$\Diamond$\par\bigskip}
\end{examplenorm}

\begin{examplenorm}\label{Example Asymptotic Distribution Parametric - Pareto}
Consider the subset $\Theta:= (1,\infty)$ of $\Upsilon:=\R$, and let $F_{a}$ be the distribution function of the Pareto distribution with {\em known} location parameter $\overline{c}>1$ and unknown tail-index $a\in\Theta$ as defined in Example \ref{Example Consitency Parametric - Pareto}. Moreover consider the map $\mathfrak{F}:\Theta\rightarrow\F_1(\subseteq\bL_0)$ defined by $\mathfrak{F}(a):=F_{a}$. An easy exercise shows that we may apply Lemma \ref{Asymptotc Distribution cor - parametric} to conclude that for every $a_{0}\in\Theta$ the map $\mathfrak{F}$ is Hadamard differentiable at $a_{0}\in\Theta$ (tangentially to the whole space $\Upsilon=\R$) with trace $\bL_1$ and Hadamard derivative $\dot \mathfrak{F}_{a_{0}}: \R\rightarrow\bL_1$ given by
\begin{equation}\label{Example Asymptotic Distribution Parametric - Pareto - derivative}
    \dot\mathfrak{F}_{a_{0}}(y)(x):=
    \left\{
    \begin{array}{lll}
         y\,\log(\overline{c}/x)\,(\overline{c}/x)^{a_{0}} & , & x> \overline{c}\\
        0 & , & x\le \overline{c}
    \end{array}
    \right..
\end{equation}

We may also verify easily that the family $\{{\rm Par}_{a,\overline{c}}:a\in\Theta\}$ satisfies the assumptions of Theorem 6.2.6 in \cite{LehmannCasella1998}. Therefore the maximum likelihood estimator $\widehat{a}_{n}$ given by (\ref{ML Pareto}) in the corresponding infinite statistical product model satisfies
\begin{equation}\label{asymptotic normality ML Pareto}
    \sqrt{n}(\widehat{a}_n- a_{0})\,\leadsto\,Y_{a_{0}} \qquad\mbox{in $(\Upsilon,{\cal B}(\Upsilon))$}
\end{equation}
under $\pr^{a_{0}}$ for every $a_{0}\in\Theta$, where $Y_{a_{0}}$ is centered normal with $\vari[Y_{a_{0}}] = a_{0}^{2}$. For the Hill estimators $\widehat{a}^{H}_{n,k_{n}}$ as defined by (\ref{Hill Pareto}) Theorem 2 in \cite{Hall1982} shows that, if $k_{n}\to\infty$ and $k_{n}={\cal O}(n^\gamma)$ for some $\gamma<1$,
\begin{equation}
\label{asymptotic normality Hill Pareto}
    \sqrt{k_{n}} (\widehat{a}_{n,k_n}^H - a_{0})\,\leadsto\,Y_{a_{0}} \qquad\mbox{in $(\Upsilon,{\cal B}(\Upsilon))$}
\end{equation}
under $\pr^{a_{0}}$ for every $a_{0}\in\Theta$, where $Y_{a_{0}}$ is the same as in (\ref{asymptotic normality ML Pareto}). That is, up to the rate of convergence, the asymptotic of the maximum likelihood estimator $\widehat a_n$ is the same as the asymptotic of the Hill estimator $\widehat a_{n,k_n}^H$.

Now we may apply Theorem \ref{Asymptotc Distribution thm - parametric} to obtain that both $\sqrt{n}({\cal R}_\alpha(F_{\widehat a_n})-{\cal R}_\alpha(F_{a_0}))$ and $\sqrt{k_n}({\cal R}_\alpha(F_{\widehat a_{n,k_n}^H})-{\cal R}_\alpha(F_{a_0}))$ converge in distribution to $Z:=\dot\cR_{\alpha;F_{a_0}}(\dot\mathfrak{F}_{a_0}(Y_{a_0}))$. In Section \ref{Proof of Example Asymptotic Distribution Parametric - Pareto} it is shown that the limit $Z$ is centered normal with variance
\begin{equation}\label{Example Asymptotic Distribution Parametric - Pareto - 10}
    \frac{\overline{c}^{2}}{(1- a_{0})^{4}\{(1 - 2\alpha) F_{a_{0}}(\cR_{\alpha}(F_{a_{0}})) + \alpha\}^{2}}\,\varphi(a_{0},\overline{c})^{2},
\end{equation}
where $\varphi(a_{0},\overline{c}):=(\cR_{\alpha}(F_{a_{0}})/\overline{c})^{1-a_{0}}\,(1 - (1-a_{0})\log(\cR_{\alpha}(F_{a_{0}})/\overline{c}))\,(1- 2\alpha) + \alpha - 1$.
{\hspace*{\fill}$\Diamond$\par\bigskip}
\end{examplenorm}


\section{Proof of Theorem \ref{main theorem}}\label{proof of main theorem}

For every $F\in\F_1$ the map ${\cal U}_\alpha(F)(\cdot)$ given by (\ref{Def Mapping cal U alpha}) is real-valued, continuous, strictly decreasing, satisfies $\lim_{m\to\pm\infty}{\cal U}(F)(m)=\mp\infty$, and may be represented as
\begin{eqnarray}
    {\cal U}_\alpha(F)(m)
    & = & -(1-\alpha)\int_{(-\infty,m)}\hspace*{-0.3cm} F(x)\,\ell(dx)\,+\,\alpha\int_{(m,\infty)}\hspace*{-0.3cm} (1-F(x))\,\ell(dx)\label{representation int by parts - eq1}\\
    & = & -(1-\alpha)\int_{(-\infty,0)}\hspace*{-0.3cm} F(x+m)\,\ell(dx)\,+\,\alpha\int_{(0,\infty)}\hspace*{-0.3cm} (1-F(x+m))\,\ell(dx).\label{representation int by parts - mathbb - eq}
\end{eqnarray}
This follows by Lemma \ref{representation int by parts - mathbb} and ensures that the functional ${\cal R}_\alpha$ defined by (\ref{def risk functional}), i.e.\ the mapping
$\F_1\rightarrow\R$, $F\mapsto{\cal R}_\alpha(F):={\cal U}_\alpha(F)^{-1}(0)$, is well defined.


\subsection{Auxiliary lemmas}

\begin{lemma}\label{quasi HD of U}
Let $F\in\F_1$. Moreover let $(v,(v_n),(\varepsilon_n))$ be any triplet with $v\in\bL_1$, $(\varepsilon_n)\subseteq(0,\infty)$ satisfying $\varepsilon_n\to 0$, and $(v_n)\subseteq\bL_1$ satisfying $\|v_n-v\|_{1,\ell}\to 0$ as well as $F+\varepsilon_nv_n \in\F_1$ for every $n\in\N$. Then the following two assertions hold:

(i) We have
$$
    \lim_{n\to\infty}\,\sup_{m\in\R}\,\Big|\frac{{\cal U}_\alpha(F+\varepsilon_nv_n)(m)-{\cal U}_\alpha(F)(m)}{\varepsilon_n}-\dot{\cal U}_{\alpha}(v)(m)\Big|\,=\,0,
$$
where
\begin{equation}\label{Def HD of cal U}
    \dot{\cal U}_{\alpha}(v)(m)\,:=\,-(1-\alpha)\int_{(-\infty,0)} v(x+m)\,\ell(dx)\,-\,\alpha\int_{(0,\infty)} v(x+m)\,\ell(dx).
\end{equation}

(ii) For any $\varepsilon>0$ there is some $n_0=n_0(\varepsilon)\in\N$ such that for every $n\ge n_0$ the value ${\cal R}_\alpha(F+\varepsilon_nv_n)$ lies in the open interval $({\cal R}_\alpha(F)-\varepsilon,{\cal R}_\alpha(F)+\varepsilon)$.
\end{lemma}

\begin{proof}
(i): In view of (\ref{representation int by parts - mathbb - eq}) and (\ref{Def HD of cal U}), we have
\begin{eqnarray*}
    \lefteqn{\Big|\frac{{\cal U}_\alpha(F+\varepsilon_nv_n)(m)-{\cal U}_\alpha(F)(m)}{\varepsilon_n}\,-\,\dot{\cal U}_{\alpha}(v)(m)\Big|}\\
    & \le & (1-\alpha)\int_{(-\infty,0)}|v_n(x+m)-v(x+m)|\,\ell(dx)\\
    & & +\,\alpha\,\int_{(0,\infty)}|v_n(x+m)-v(x+m)|\,\ell(dx)~\,\le~\,\|v_n-v\|_{1,\ell}
\end{eqnarray*}
for every $m\in\R$. This gives the claim.

(ii): The assumption $\|v_n-v\|_{1,\ell}\to 0$ implies $\|(F+\varepsilon_nv_n)-F\|_{1,\ell}\to 0$, and therefore the claim is an immediate consequence of Theorem \ref{1 weak continuity of R alpha}.
\end{proof}

%

For $-\infty < a < b < \infty$ let $\bB[a,b]$ denote the space of all bounded Borel measurable functions $f:[a,b]\rightarrow\R$. The space $\bB[a,b]$ will be equipped with the sup-norm $\|\cdot\|_\infty$. Furthermore, let $\bB_{\downarrow\,,0}[a,b]$ be the set of all non-increasing $f\in\bB[a,b]$ satisfying the inequalities $f(a)\geq 0\geq f(b)$. Then the mapping
$$
    {\cal I}_{a,b}:\bB_{\downarrow\,,0}[a,b]\longrightarrow\R,\qquad f\longmapsto f^{\rightarrow}(0)
$$
is well-defined, where $f^{\rightarrow}(0) := \sup\{x\in [a,b]:f(x)>0\}$ with $\sup\emptyset := a$. Further, for $x_0\in(a,b)$ let $\bB_{{\rm c},{x_{0}}}[a,b]$ denote the linear subspace of $\bB[a,b]$ consisting of all elements of $\bB[a,b]$ which are continuous at $x_{0}$. In the following lemma we employ the notion of tangential Hadamard differentiability in the classical sense as defined in \cite{Gill1989}; see also Definition \ref{definition quasi hadamard} and Remark \ref{definition quasi hadamard - remark} below.

\begin{lemma}\label{HDdiffbarkeitQuantile}
Let $-\infty < a < b < \infty,$ and let $f\in\bB_{\downarrow\,,0}[a,b]$ be differentiable at some $x_{0}\in (a,b)\cap f^{-1}(\{0\})$ with strictly negative derivative $f'(x_{0})$. Then ${\cal I}_{a,b}$ is Hadamard differentiable at $f$ tangentially to $\bB_{{\rm c},{x_{0}}}[a,b]$ with Hadamard derivative $\dot{\cal I}_{a,b;f}: \bB_{{\rm c},x_{0}}[a,b]\rightarrow\R$ given by
$$
    \dot{\cal I}_{a,b;f}(w):=-\frac{w(x_{0})}{f'(x_{0})}\,,\qquad w\in\bB_{{\rm c},x_{0}}[a,b].
$$
\end{lemma}

\begin{proof}
The following proof is inspired by the proof of Lemma 3.9.20 in \cite{van der Vaart Wellner 1996}, where an analogous result has been shown under the additional assumption that $f$ is \cadlag. For the convenience of the reader we give a detailed argumentation.

First of all note that $\dot{\cal I}_{a,b;f}(\cdot)$ is obviously continuous w.r.t.\ the sup-norm $\|\cdot\|_\infty$ (and linear). Now, let $(w_{n})\subseteq\bB_{{\rm c},{x_{0}}}[a,b]$ and $(\varepsilon_{n})\subseteq(0,\infty)$ be any sequences such that $f+ \varepsilon_{n} w_{n}\in\bB_{\downarrow\,,0}[a,b]$ for all $n\in\N$, $\varepsilon_{n}\to 0$, and $\|w_{n}- w\|_\infty\to 0$ for some $w\in\bB_{{\rm c},{x_{0}}}[a,b]$. It remains to show
\begin{equation}\label{HDdiffbarkeitQuantile-PROOF-10}
    \lim_{n\to\infty}\Big|\dot{\cal I}_{a,b;f}(w)-\frac{{\cal I}_{a,b}(f+\varepsilon_nw_n)-{\cal I}_{a,b}(f)}{\varepsilon_n}\Big|\,=\,0.
\end{equation}
For (\ref{HDdiffbarkeitQuantile-PROOF-10}) it suffices to show
\begin{eqnarray}
    - w(x_{0})
    & \le & \liminf_{n\to\infty}\,f'(x_{0})\,\frac{\cI_{a,b}(f + \varepsilon_{n} w_{n}) - \cI_{a,b}(f)}{\varepsilon_{n}}\,,\label{HDdiffbarkeitQuantile-PROOF-20}\\
    - w(x_{0})
    & \ge & \limsup_{n\to\infty}\,f'(x_{0})\,\frac{\cI_{a,b}(f + \varepsilon_{n} w_{n}) - \cI_{a,b}(f)}{\varepsilon_{n}}\,.\label{HDdiffbarkeitQuantile-PROOF-30}
\end{eqnarray}
Since $f$ is non-increasing with $f'(x_{0})<0$, we have $f^{-1}(\{0\})=\{x_0\}$. That is, $f(x_0)=0$, $f(x)>0$ for $x\in [a,x_{0})$, and $f(x) < 0$ for $x\in (x_{0},b]$. Furthermore, in view of $\|f-(f+\varepsilon_{n}w_{n})\|_{\infty} = \varepsilon_{n}\|w_{n}\|_{\infty}\to 0$, we may assume without loss of generality that $\cI_{a,b}(f +\varepsilon_{n} w_{n})\in (a,b)$ for all $n\in\N$. Thus we may select a sequence $(\gamma_{n})$ in $(0,1)$ such that $\gamma_{n}\le\varepsilon_{n}^{2}$,
$$
    a < \cI_{a,b}(f+\varepsilon_n w_n)-\gamma_{n},\quad \cI_{a,b}(f+\varepsilon_n w_n) + \gamma_{n} < b,\quad\mbox{and}\quad\cI_{a,b}(f+\varepsilon_n w_n)\pm\gamma_{n}\not=x_{0}
$$
for all $n\in\N$. By the definition of $\cI_{a,b}$ we then have for every $n\in\N$
\begin{equation}\label{BasisUngleichungen}
    (f+\varepsilon_n w_n)(\cI_{a,b}(f+\varepsilon_n w_n) - \gamma_{n})\,\ge\,0\,\ge\,(f+\varepsilon_n w_n)(\cI_{a,b}(f+\varepsilon_n w_n) + \gamma_{n}).
\end{equation}
In Step 2 ahead we will show that also the following three assertions hold:
\begin{equation}\label{KonvergenzInverse}
    \lim_{n\to\infty}\cI_{a,b}(f+\varepsilon_n w_n)\,=\,x_{0}\,=\,\cI_{a,b}(f),
\end{equation}
\begin{equation}\label{boundedness}
    \sup_{n\in\N}\Big|\frac{\cI_{a,b}(f + \varepsilon_{n} w_{n}) - \cI_{a,b}(f)}{\varepsilon_{n}}\Big|\,<\,\infty,
\end{equation}
\begin{equation}\label{KonvergenzTangentialvektoren}
    \lim_{n\to\infty}w_{n}\big(\cI_{a,b}(f+\varepsilon_n w_n) - \gamma_n\big)\,=\,w(x_{0})\,=\,\lim_{n\to\infty}w_{n}\big(\cI_{a,b}(f+\varepsilon_n w_n) + \gamma_n\big).
\end{equation}
Before, we will show in Step 1 that (\ref{BasisUngleichungen})--(\ref{KonvergenzTangentialvektoren}) imply (\ref{HDdiffbarkeitQuantile-PROOF-20})--(\ref{HDdiffbarkeitQuantile-PROOF-30}).

{\em Step 1}. Let
$$
    a_{n}:=\frac{f\big(\cI_{a,b}(f + \varepsilon_{n} w_{n}) + \gamma_{n} \big) - f(x_{0})}{\cI_{a,b}(f + \varepsilon_{n} w_{n}) + \gamma_{n} - x_{0}}\quad\mbox{ and }\quad
    b_{n}:=\frac{f\big(\cI_{a,b}(f + \varepsilon_{n} w_{n}) - \gamma_{n} \big) - f(x_{0})}{\cI_{a,b}(f+\varepsilon_{n} w_{n}) - \gamma_{n} - x_{0}}\,.
$$
By $x_0=\cI_{a,b}(f)$ and $f(x_0)=0$ we have
\begin{eqnarray*}
    -b_n\,\frac{\cI_{a,b}(f+\varepsilon_nw_n)-\cI_{a,b}(f)}{\varepsilon_n}+b_n\,\frac{\gamma_n}{\varepsilon_n}
    & = & -b_n\,\frac{\cI_{a,b}(f+\varepsilon_nw_n)-\gamma_n-x_0}{\varepsilon_n}\\
    & = & -\frac{f(\cI_{a,b}(f+\varepsilon_nw_n)-\gamma_n)-f(x_0)}{\varepsilon_n}\\
    & = & -\frac{f(\cI_{a,b}(f+\varepsilon_nw_n)-\gamma_n)}{\varepsilon_n}\,.
\end{eqnarray*}
Moreover, by (\ref{BasisUngleichungen}) we have
$$
    -w_n(\cI_{a,b}(f+\varepsilon_nw_n)-\gamma_n)\,\le\,\frac{f(\cI_{a,b}(f+\varepsilon_nw_n)-\gamma_n)}{\varepsilon_n}\,.
$$
Hence,
\begin{eqnarray}
    & & - (b_{n} - f'(x_{0}))\,\frac{\cI_{a,b}(f + \varepsilon_{n} w_{n}) - \cI_{a,b}(f)}{\varepsilon_{n}} + b_{n}\,\frac{\gamma_{n}}{\varepsilon_{n}} - w_{n}(\cI_{a,b}(f + \varepsilon_{n} w_{n}) - \gamma_{n})\nonumber\\
    & \le & f'(x_{0})\,\frac{\cI_{a,b}(f + \varepsilon_{n} w_{n}) - \cI_{a,b}(f)}{\varepsilon_{n}}\,.\label{HDdiffbarkeitQuantile-PROOF-100}
\end{eqnarray}
Similarly we obtain
\begin{eqnarray}
    & & f'(x_{0})\,\frac{\cI_{a,b}(f + \varepsilon_{n} w_{n}) - \cI_{a,b}(f)}{\varepsilon_{n}}\label{HDdiffbarkeitQuantile-PROOF-200}\\
    & \le & - (a_{n} - f'(x_{0}))\,\frac{\cI_{a,b}(f + \varepsilon_{n} w_{n}) - \cI_{a,b}(f)}{\varepsilon_{n}} - a_{n}\,\frac{\gamma_{n}}{\varepsilon_{n}} - w_{n}(\cI_{a,b}(f + \varepsilon_{n} w_{n}) + \gamma_{n}).\nonumber
\end{eqnarray}
By differentiability of $f$ at $x_{0}$ and (\ref{KonvergenzInverse}) we obtain that both $(a_{n} - f'(x_{0}))$ and $(b_{n} - f'(x_{0}))$ converge to zero as $n\to\infty$. Along with (\ref{boundedness}) we can conclude that
\begin{eqnarray}
    \lim_{n\to\infty}\,(a_{n} - f'(x_{0}))\,\frac{\cI_{a,b}(f + \varepsilon_{n} w_{n}) - \cI_{a,b}(f)}{\varepsilon_{n}} & = & 0,\nonumber\\
    \lim_{n\to\infty}\,(b_{n} - f'(x_{0}))\,\frac{\cI_{a,b}(f + \varepsilon_{n} w_{n}) - \cI_{a,b}(f)}{\varepsilon_{n}} & = &  0,\label{HDdiffbarkeitQuantile-PROOF-300}
\end{eqnarray}
and along with the choice of $(\gamma_{n})$ we can conclude that
\begin{eqnarray}
    \lim_{n\to\infty}a_{n}\,\frac{\gamma_{n}}{\varepsilon_{n}}\,=\,f'(x_{0})\,\lim_{n\to\infty}\frac{\gamma_{n}}{\varepsilon_{n}}\,=\,0,\nonumber\\
    \lim_{n\to\infty}b_n\,\frac{\gamma_{n}}{\varepsilon_{n}}\,=\,f'(x_{0})\,\lim_{n\to\infty}\frac{\gamma_{n}}{\varepsilon_{n}}\,=\,0.\label{HDdiffbarkeitQuantile-PROOF-400}
\end{eqnarray}
Now, (\ref{HDdiffbarkeitQuantile-PROOF-100})--(\ref{HDdiffbarkeitQuantile-PROOF-200}) along with (\ref{KonvergenzTangentialvektoren}), (\ref{HDdiffbarkeitQuantile-PROOF-300}), and (\ref{HDdiffbarkeitQuantile-PROOF-400}) imply (\ref{HDdiffbarkeitQuantile-PROOF-20})--(\ref{HDdiffbarkeitQuantile-PROOF-30}).

{\em Step 2}. It remains to show (\ref{KonvergenzInverse})--(\ref{KonvergenzTangentialvektoren}). First we show (\ref{KonvergenzInverse}). The inequalities in (\ref{BasisUngleichungen}) imply
$$
    f(\cI_{a,b}(f+\varepsilon_n w_n) - \gamma_{n})\ge- \varepsilon_{n} \|w_{n}\|_{\infty}\quad\mbox{ and }\quad
    f(\cI_{a,b}(f+\varepsilon_n w_n) + \gamma_{n})\le\varepsilon_{n} \|w_{n}\|_{\infty}.
$$
In particular,
\begin{equation}\label{Grenzwerte}
    \liminf_{n\to\infty} f(\cI_{a,b}(f+\varepsilon_n w_n) - \gamma_{n})\,\ge\,0\,\ge\,\limsup_{n\to\infty}f(\cI_{a,b}(f+\varepsilon_n w_n) + \gamma_{n}).
\end{equation}
Now let $\varepsilon > 0$ be such that $a < x_{0}- \varepsilon < x _{0} + \varepsilon < b$. Then $f(x_{0} + \varepsilon/2) < 0 < f(x_{0} - \varepsilon/2)$, and in view of
(\ref{Grenzwerte}) we may find some $n_{0}\in\N$ such that for every $n\geq n_{0}$ we have $\gamma_{n} < \varepsilon/2$ and
\begin{eqnarray*}
    &&
    f\big(\cI_{a,b}(f+\varepsilon_n w_n) - \gamma_{n}\big)\,>\,- |f(x_{0} + \varepsilon/2)|\,=\,f(x_{0} + \varepsilon/2),\\
    &&
    f\big(\cI_{a,b}(f +\varepsilon_n w_n) + \gamma_{n}\big)\,<\,f(x_{0} - \varepsilon/2).
\end{eqnarray*}
Since $f$ is non-increasing, this means that $|\cI_{a,b}(f+\varepsilon_n w_n) - x_{0}| < \varepsilon$ for all $n\geq n_{0}$. That is, (\ref{KonvergenzInverse}) indeed holds.

Next, (\ref{KonvergenzTangentialvektoren}) is an immediate consequence of (\ref{KonvergenzInverse}), $\|w_{n} - w\|_{\infty}\to 0$, $\gamma_n\to0$, and the continuity of $w$ at $x_{0}$.

Finally we will show by way of contradiction that (\ref{boundedness}) holds. So let us first assume that $(\cI(f_{0} + \varepsilon_{i(n)} w_{i(n)}) - x_{0})/\varepsilon_{i(n)}\to - \infty$ holds for some subsequence $(i(n))\subseteq(n)$. By $f_{0}(x_{0}) = 0$ we have
\begin{eqnarray*}
    \lefteqn{\frac{(f + \varepsilon_{i(n)} w_{i(n)})\big(\cI_{a,b}(f + \varepsilon_{i(n)} w_{i(n)}) + \gamma_{i(n)}\big)}{\varepsilon_{i(n)}}}\\
    & = & \frac{(f + \varepsilon_{i(n)} w_{i(n)})\big(\cI_{a,b}(f + \varepsilon_{i(n)} w_{i(n)}) + \gamma_{i(n)}\big) - f(x_{0})}{\varepsilon_{i(n)}}\\
    & = & a_{i(n)}\,\frac{\cI_{a,b}(f + \varepsilon_{i(n)} w_{i(n)}) - x_{0}}{\varepsilon_{i(n)}} + a_{i(n)}\,\frac{\gamma_{i(n)}}{\varepsilon_{i(n)}} + w_{i(n)}(\cI_{a,b}(f +  \varepsilon_{i(n)}w_{i(n)}) + \gamma_{i(n)})
\end{eqnarray*}
for every $n\in\N$. Since $f$ is differentiable at $x_{0}$ with strictly negative derivative, we obtain from (\ref{KonvergenzInverse})
$$
    \lim_{n\to\infty}\,a_{i(n)}\,\frac{\cI_{a,b}(f + \varepsilon_{i(n)} w_{i(n)}) - x_{0}}{\varepsilon_{i(n)}}=\infty
$$
and
$$
    \lim_{n\to\infty}a_{i(n)}\,\frac{\gamma_{i(n)}}{\varepsilon_{i(n)}} =
    f_{0}'(x_{0})\,\lim_{n\to\infty}\frac{\gamma_{i(n)}}{\varepsilon_{i(n)}} = 0.
$$
Therefore in view of (\ref{KonvergenzTangentialvektoren}), we may conclude
\begin{equation}\label{erster Widerspruch}
    \lim_{n\to\infty}\frac{(f + \varepsilon_{i(n)} w_{i(n)})\big(\cI_{a,b}(f + \varepsilon_{i(n)} w_{i(n)}) + \gamma_{i(n)}\big)}{\varepsilon_{i(n)}} = \infty
\end{equation}
which contradicts (\ref{BasisUngleichungen}). In a similar way we obtain a contradiction when supposing that $(\cI_{a,b}(f + \varepsilon_{j(n)} w_{j(n)}) - x_{0})/\varepsilon_{j(n)}\to\infty$ for some subsequence $(j(n))\subseteq(n)$, using $(b_{j(n)})$.
\end{proof}


\subsection{Main part of the proof}

Let $F\in\F_1$ and assume that it is continuous at ${\cal R}_\alpha(F)$. First of all note that the functional $\dot{\cal R}_{\alpha,F}:\bL_1\to\R$ defined by (\ref{def of limit of emp process}) is easily seen to be continuous.

Now, let $(v_{n})\subseteq\bL_1$ and $(\varepsilon_{n})\subseteq(0,\infty)$ be any sequences such that $F + \varepsilon_{n} v_{n}\in \F_{1}$ for all $n\in\N$, $\varepsilon_{n}\to 0$, and $\|v_{n}- v\|_{1,\ell}\to 0$ for some $v\in\bL_1$. In view of part (ii) of Lemma \ref{Def HD of cal U} we may assume without loss of generality that $({\cal R}_{\alpha}(F + \varepsilon_{n} v_{n}))$ is a sequence in $[a,b]$ with $a := {\cal R}_{\alpha}(F) - \varepsilon$ and $b := {\cal R}_{\alpha}(F) + \varepsilon$ for some $\varepsilon > 0$. Setting $f := {\cal U}_{\alpha}(F)|_{[a,b]}$ and $f_{n} := {\cal U}_{\alpha}(F + \varepsilon_{n} v_{n})|_{[a,b]}$ for $n\in\N$, this means that $(f_{n})_{n\in\N_{0}}$ is a sequence in $\bB_{\downarrow,\,0}[a,b]$, and
\begin{eqnarray*}
    \frac{{\cal R}_{\alpha}(F + \varepsilon_{n} v_{n}) - {\cal R}_{\alpha}(F)}{\varepsilon_{n}}
    \,=\,
    \frac{{\cal I}_{a,b}(f_{n}) - {\cal I}_{a,b}(f)}{\varepsilon_{n}}
    \,=\,
    \frac{{\cal I}_{a,b}(f + \varepsilon_{n}\frac{f_{n} -f}{\varepsilon_{n}}) - {\cal I}_{a,b}(f)}{\varepsilon_{n}}
\end{eqnarray*}
for all $n\in\N$. Taking the identity $(1-2\alpha)F({\cal R}_\alpha(F))\,+\,\alpha = (1-\alpha)F({\cal R}_\alpha(F))+\alpha(1-F({\cal R}_\alpha(F)))$ and the definition of $\dot{\cal R}_{\alpha,F}$ by (\ref{def of limit of emp process}) into account, it thus remains to show
\begin{equation}\label{Main part of the proof - 10}
    \lim_{n\to\infty}\,\frac{{\cal I}_{a,b}(f + \varepsilon_{n}\frac{f_{n} -f}{\varepsilon_{n}}) - {\cal I}_{a,b}(f)}{\varepsilon_{n}}
    \,=\,
    \frac{\dot{{\cal U}}_{\alpha}(v)({\cal R}_{\alpha}(F))}{(1-\alpha) F({\cal R}_{\alpha}(F)) + \alpha(1 - F({\cal R}_{\alpha}(F)))}\,,
\end{equation}
where $\dot{\cal U}_{\alpha}$ is given by (\ref{Def HD of cal U}).

We intend to apply Lemma \ref{HDdiffbarkeitQuantile} in order to verify (\ref{Main part of the proof - 10}). By part (i) of Lemma \ref{quasi HD of U} we have
$$
    \lim_{n\to\infty}\,\sup_{m\in[a,b]}\,|w_n(m)-w(m)|\,=\,0
$$
for
$$
    w_n(\cdot):=\frac{f_{n}(\cdot)- f(\cdot)}{\varepsilon_n}\quad\mbox{ and }\quad w(\cdot):=\dot{\cal U}_{\alpha}(v)(\cdot).
$$
According to (\ref{representation int by parts - eq1}) we have
$$
    f(m)\,=\,-(1-\alpha)\int_{(-\infty,m)} F(x)\,\ell(dx)\,+\,\alpha\int_{(m,\infty)} (1-F(x))\,\ell(dx)\quad\mbox{ for all }m\in [a,b].
$$
Since $F$ is continuous at ${\cal R}_{\alpha}(F)$ by assumption, we may apply the second fundamental theorem of calculus for regulated functions to conclude that the function $f$ is differentiable at $x_0:={\cal R}_\alpha(F)={\cal I}_{a,b}(f)\in (a,b)$ with derivative

$$
    f'(x_0)\,=\,-(1-\alpha) F({\cal R}_{\alpha}(F)) - \alpha \big(1 - F({\cal R}_{\alpha}(F))\big)\,<\,0.
$$
Below we will show that $w_n$ and $w$ are continuous on $[a,b]$. So Lemma \ref{HDdiffbarkeitQuantile} implies (\ref{Main part of the proof - 10}).

It remains to show the continuity of $w_n$ and $w$. In view of the definition of $w(\cdot)=\dot{\cal U}_{\alpha}(v)$ due to (\ref{Def HD of cal U}), we obtain by change of variable formula
\begin{eqnarray*}
    |w(m_1)-w(m_2)|
    & = &
    |\dot{\cal U}_{\alpha}(v)(m_{1}) - \dot{\cal U}_{\alpha}(v)(m_{2})|\\
    & \le & (1 - \alpha) \Big|\int_{(m_1,m_2)}v(x)\,\ell(dx)\Big|\,+\,\alpha \Big|\int_{(m_1,m_2)}v(x)\,\ell(dx)\Big|\\
    & \le & \Big|\int_{(m_1,m_2)}v(x)\,\ell(dx)\Big|.
\end{eqnarray*}
Since $v$ as an element of $\bL_1$ is Lebesgue integrable, it follows that $w$ is continuous on $[a,b]$. Further, in view of (\ref{representation int by parts - eq1}) we have
$$
    w_{n}(m)\,=\,(1-\alpha)\int_{(-\infty,m)} v_{n}(x)\,\ell(dx)\,+\,\alpha\int_{(m,\infty)} v_{n}(x)\,\ell(dx) \quad\mbox{ for all }m\in [a,b],
$$
and thus we may show continuity of $w_{n}$ in the same way as we have done for $w$. This completes the proof of Theorem \ref{main theorem}.
\hfill$\Box$


\section{Remaining proofs}\label{Sec Remaining proofs}

\subsection{Proof of Remark \ref{example for cond of main theorem coroll}}\label{Proof of example for cond of main theorem coroll}

For (i) note that finiteness of the integral $\int \phi^2\,dF$ implies that there exists a constant $C>0$ such that $F(y)\le C\phi(y)^{-2}$ for $y<0$ and $1-F(y)\le C\phi(y)^{-2}$ for $y>0$. In view of $\int1/\phi\,d\ell<\infty$, it follows that the integral $\int\sqrt{F(1-F)}\,d\ell$ is finite. Assertion (ii) is trivial. Condition (\ref{main theorem coroll - eq - 20}) implies that $\widetilde{\alpha}(n)\le 1/2$ for sufficiently large $n$, and $Q_{|X_{1}|}(u) u^{-1/2}\geq Q_{|X_{1}|}(1/2)/\sqrt{2}$ holds for $u\in (0,1/2)$ anyway. Thus assertion (iii) follows easily. Concerning assertion (iv) note that by application of Fubini's theorem we may observe
\begin{eqnarray*}
    \frac{1}{2}\int_{(0,\infty)}\widetilde{\alpha}(n)^{1/2}\wedge \pr[|X_1|>x]^{1/2}\,\ell(dx)
    & = &
    \int_{(0,\infty)} \int_{(0,\pr[|X_1|>x])}\eins_{(0,\widetilde{\alpha}(n))}(u)\,u^{-1/2}\,\ell(du)\,\ell(dx)\\
    & = &
    \int_{(0,\widetilde\alpha(n))}Q_{|X_1|}(u)\,u^{-1/2}\,\ell(du).
\end{eqnarray*}
Finally, assertion (v) is a consequence of (iv) and the equivalence of the integrability conditions $\int\sqrt{F(1-F)}\,d\ell<\infty$ and $\int_{(0,\infty)}\pr[|X_1|>x]\,\ell(dx)<\infty$. 
\hfill$\Box$


\subsection{Proof of Example \ref{Example Consitency Parametric - Lognormal}}\label{Proof of Example Consitency Parametric - Lognormal}

For every $(m_0,s_0^2),(m_n,s_n^2)\in\Theta$ and $f\in{\cal C}_1$ we have
\begin{eqnarray}
    \lefteqn{\Big|\int f\,dF_{(m_n,s^2_n)}-\int f\,dF_{(m_0,s^2_0)}\Big|}\nonumber\\
    & \le & \int_{(0,\infty)} \frac{|f(x)|}{|x|}\,\Big|\frac{1}{\sqrt{2\pi s_n^2}}\,e^{-\{\log(x)-m_n\}^2/\{2s_n^2\}}-\frac{1}{\sqrt{2\pi s_0^2}}\,e^{-\{\log(x)-m_0\}^2/\{2s_0^2\}}\Big|\,\ell(dx)\nonumber\\
    & \le & C_f\int_{(0,\infty)} \Big|\frac{1}{\sqrt{2\pi s_n^2}}\,e^{-\{\log(x)-m_n\}^2/\{2s_n^2\}}-\frac{1}{\sqrt{2\pi s_0^2}}\,e^{-\{\log(x)-m_0\}^2/\{2s_0^2\}}\Big|\,\ell(dx)
    \label{Proof of Example Consitency Parametric - Lognormal - EQ 10}
\end{eqnarray}
for some finite constant $C_f>0$ depending on $f$. If $(m_n,s_n^2)\to(m_0,s_0^2)$, then
$$
    \lim_{n\to\infty}\,\frac{1}{\sqrt{2\pi s_n^2}}\,e^{-\{\log(x)-m_n\}^2/\{2s_n^2\}}=\frac{1}{\sqrt{2\pi s_0^2}}\,e^{-\{\log(x)-m_0\}^2/\{2s_0^2\}}\quad\mbox{ for all }x\in\R
$$
and
\begin{eqnarray*}
    \lim_{n\to\infty}\int_{(0,\infty)}\frac{1}{\sqrt{2\pi s_n^2}}\,e^{-\{\log(x)-m_n\}^2/\{2s_n^2\}}\,\ell(dx)
    &=&
    \lim_{n\to\infty}\int_{\R}x\,dF_{(m_{n},s_{n}^{2})}(x)\\
    &=&
    \lim_{n\to\infty}e^{m_{n} + s_{n}^{2}/2}\\
    &=&
    e^{m_{0} + s_{0}^{2}/2}\\
    &=&
    \int_{\R}x\,dF_{(m_{0},s_{0}^{2})}(x)\\
    &=&
    \int_{(0,\infty)}\frac{1}{\sqrt{2\pi s_0^2}}\,e^{-\{\log(x)-m_0\}^2/\{2s_0^2\}}\,\ell(dx).
\end{eqnarray*}
Since  $\frac{1}{\sqrt{2\pi s_n^2}}\,e^{-\{\log(x)-m_n\}^2/\{2s_n^2\}}\geq 0$ for all $x\in\R$ and $n\in\N$, Lemma 21.6 in \cite{Bauer2001} yields
$$
    \lim_{n\to\infty}\int_{(0,\infty)} \Big|\frac{1}{\sqrt{2\pi s_n^2}}\,e^{-\{\log(x)-m_n\}^2/\{2s_n^2\}}-\frac{1}{\sqrt{2\pi s_0^2}}\,e^{-\{\log(x)-m_0\}^2/\{2s_0^2\}}\Big|\,\ell(dx)=0.
$$
Along with (\ref{Proof of Example Consitency Parametric - Lognormal - EQ 10}) and Remark \ref{remark on Consistency thm - parametric} this shows that the mapping $(m,s^2)\mapsto F_{(m,s^2)}$ is $1$-weakly sequentially continuous at every $(m_0,s_0^2)\in\Theta$.
\hfill$\Box$


\subsection{Proof of Lemma \ref{Asymptotc Distribution cor - parametric}}\label{vereinfachtes Kriterium}

Consider any triplet $(\tau,(\tau_n),(\varepsilon_n))$ with $\tau\in \R^{d}$, $(\tau_n)\subseteq\R^{d}$ satisfying $(\theta_{0} + \varepsilon_{n}\tau_{n})\subseteq\Theta$ as well as $\|\tau_n-\tau\|\to 0$, and $(\varepsilon_n)\subseteq(0,\infty)$ satisfying $\varepsilon_n\to 0$. Since ${\cal V}$ is an open subset of $\R^{d}$ containing $\theta_{0}$, we may assume without loss of generality that $\theta_{0} + \varepsilon_{n}\tau_{n}\in{\cal V}$ for every $n\in\N$. Since $\mathfrak{f}(\,\cdot\,,x)$ is continuously differentiable at $\theta_{0}$ for every $x\in\R$, we may conclude
\begin{eqnarray*}
    \lim_{n\to\infty}\frac{F_{\theta_0+\varepsilon_n\tau_n}(x)- F_{\theta_0}(x)}{\varepsilon_n}\,=\,\lim_{n\to\infty}\frac{\mathfrak{f}(\theta_0+\varepsilon_n\tau_n,x) - \mathfrak{f}(\theta_{0},x)}{\varepsilon_{n}}\,=\,\langle{\rm grad}_{\theta}\,\mathfrak{f}(\theta_0,x),\tau\rangle
\end{eqnarray*}
for every $x\in\R$. Moreover, by the mean value theorem in several variables,
\begin{eqnarray*}
    \Big|\frac{F_{\theta_0+\varepsilon_n\tau_n}(x)- F_{\theta_0}(x)}{\varepsilon_n}\Big|\,\le\,
    \sup_{\theta\in{\cal V}}\,\|{\rm grad}_{\theta}\,\mathfrak{f}(\theta,x)\|\,\|\tau_{n}\|\,\le\,\mathfrak{h}(x)\,\sup_{n\in\N}\|\tau_{n}\|
\end{eqnarray*}
for all $n\in\N$ and $x\in\R$. By assumption, the majorant $\mathfrak{h}$ is $\ell$-integrable. Thus an application of the dominated convergence theorem yields
$$
    \lim_{n\to\infty}\int_{\R}\Big|\frac{F_{\theta_0+\varepsilon_n\tau_n}(x)- F_{\theta_0}(x)}{\varepsilon_n}-\langle{\rm grad}_{\theta}\,\mathfrak{f}(\theta_{0},x),\tau\rangle\Big|\,\ell(dx)\,=\,0.
$$
Thus $\mathfrak{F}$ satisfies the claimed differentiability property.
\hfill$\Box$


\subsection{Proof of Example \ref{Example Asymptotic Distribution Parametric - Lognormal}}\label{Proof of Example Asymptotic Distribution Parametric - Lognormal}

For the first assertion we intend to apply Lemma \ref{Asymptotc Distribution cor - parametric}. To this end we consider the map $\mathfrak{f}:\Theta\times\R\rightarrow[0,1]$ defined by
$$
    \mathfrak{f}((m,s^2),x):= F_{(m,s^2)}(x),
$$
where $\Theta:=\R\times (0,\infty)$. For every $(m,s^{2})\in\Theta$, the distribution function $F_{(m,s^{2})}$ satisfies $F_{(m,s^{2})}(x) = 0$ if $x\leq 0$, and $F_{(m,s^{2})}(x) = \Phi_{(0,1)}((\log(x) - m)/s)$ for $x > 0$. So, obviously, for any $x\in\R$ the map $\mathfrak{f}(\,\cdot\,,x)$ is continuously differentiable on $\Theta$ with gradient
\begin{equation}\label{Gradientendarstellung}
    {\rm grad}_{(m,s^2)}\,\mathfrak{f}((m,s^2),x) =
    \left\{
    \begin{array}{lll}
        -\big(\frac{1}{s},\frac{\log(x) - m}{2 s^{3}}\big)\,\phi_{(0,1)}\big(\frac{\log(x) - m}{s}\big) & , & x> 0\\
        (0,0) & , & x\le 0
    \end{array}
    \right..
\end{equation}
Let $(m_{0},s_{0}^{2})\in\R\times (0,\infty)$, and define the map $\mathfrak{h}:\R\rightarrow\R$ by
$$
    \mathfrak{h}(x) :=
    \left\{
    \begin{array}{lll}
        0 & , & x\leq 0\\[2mm]
        \frac{1}{\sqrt{\pi(s_{0}^{2}-\delta)}}\big(1 + \frac{2}{\delta}\big) & , & 0 < x\leq e^{m_{0} - 2\delta}\\[2mm]
        \frac{1}{\sqrt{\pi}} \Big(\frac{1}{\sqrt{s_{0}^{2} - \delta}} + \frac{C + |m_{0}| + \delta}{\sqrt{s_{0}^{2} - \delta}^{3}} \Big)& , & e^{m_{0} - 2\delta}\leq x\leq e^{m_{0} + \delta}\\[2mm]
        \frac{1}{\sqrt{\pi}}\,e^{-\frac{(\log(x) - m_{0} -\delta)^{2}}{2(s_{0}^{2} + \delta)}} \Big(\frac{1}{\sqrt{s_{0}^{2} - \delta}} + \frac{|\log(x) - m_{0} + \delta|}{\sqrt{s_{0}^{2} - \delta}^{3}} \Big) & , &x\geq e^{m_{0} + \delta}
    \end{array}
    \right..
$$

We will now show that
\begin{equation}\label{vierte Ungleichung}
    \|{\rm grad}_{(m,s^{2})}\,\mathfrak{f}((m,s^{2}),x)\|\le\mathfrak{h}(x)\mbox{ for all }((m,s^{2}),x)\in ((m_{0}-\delta,m_{0} + \delta)\times (s_{0}^{2}-\delta,s_{0}^{2}+ \delta))\times\R
\end{equation}
for some sufficiently small $\delta>0$. Let us choose $\delta > 0$ such that $s_{0}^{2}- 2\delta > 0$. In particular $(m,s^2)\in\Theta$ if $\|(m,s^2) - (m_{0},s_{0}^{2})\| < 2\delta$. For $x\in (0,e^{m_{0} - 2 \delta})$ and $(m,s^{2})\in (m_{0}-\delta,m_{0}+\delta)\times (s_{0}^{2} - \delta, s_{0}^{2} + \delta)$, we obtain
\begin{eqnarray}
    \|{\rm grad}_{(m,s^2)}\,\mathfrak{f}((m,s^2),x)\|
    & \leq & \nonumber
    \sqrt{\frac{2}{2\pi}}\,e^{-\frac{(\log(x) - m)^{2}}{2s^{2}}} \Big(\frac{1}{s} +
    \frac{|\log(x) - m|}{s^{3}} \Big)\\
    & \leq & \nonumber
    \frac{1}{\sqrt{\pi (s_{0}^{2} - \delta)}} + \frac{1}{\sqrt{\pi}}\,\frac{2s^{2} |\log(x) - m|}{s^{3} (\log(x) - m)^{2}}\\
    & \leq & \nonumber
    \frac{1}{\sqrt{\pi (s_{0}^{2} - \delta)}}\Big(1 + \frac{2}{|\log(x) - m|}\Big)\\
    & \leq& \label{erste Ungleichung}
    \frac{1}{\sqrt{\pi (s_{0}^{2} - \delta)}}\Big(1 + \frac{2}{\delta}\Big).
\end{eqnarray}
If $x > e^{m_{0} +  \delta}$ and $(m,s^{2})\in (m_{0}-\delta,m_{0}+\delta)\times (s_{0}^{2} - \delta, s_{0}^{2} + \delta)$, then the inequalities
$|\log(x) - (m_{0} + \delta)|\leq|\log(x) - m|\leq |\log(x) - (m_{0} - \delta)|$ hold, and thus
\begin{eqnarray}
    \|{\rm grad}_{(m,s^2)}\mathfrak{f}((m,s^2),x)\|
    & \leq & \nonumber
    \sqrt{\frac{2}{2\pi}}\,e^{-\frac{(\log(x) - m)^{2}}{2s^{2}}} \Big(\frac{1}{s} +
    \frac{|\log(x) - m|}{s^{3}} \Big)\\
    & \leq & \nonumber
    \frac{1}{\sqrt{\pi}}\,e^{-\frac{(\log(x) - m_{0} -\delta)^{2}}{2s^{2}}} \Big(\frac{1}{s} + \frac{|\log(x) - m_{0} + \delta|}{s^{3}} \Big)\\
    & \leq & \label{zweite Ungleichung}
    \frac{1}{\sqrt{\pi}}\,e^{-\frac{(\log(x) - m_{0} -\delta)^{2}}{2(s_{0}^{2} + \delta)}} \Big(\frac{1}{\sqrt{s_{0}^{2} - \delta}} + \frac{|\log(x) - m_{0} + \delta|}{\sqrt{s_{0}^{2} - \delta}^{3}} \Big).
\end{eqnarray}
Now, let $C := \sup_{x\in [e^{m_{0} - 2\delta},e^{m_{0} + \delta}]}|\log(x)|$. Then for $x\in [e^{m_{0} - 2\delta},e^{m_{0} + \delta}]$ and $(m,s^{2})\in (m_{0}-\delta,m_{0}+\delta)\times (s_{0}^{2} - \delta, s_{0}^{2} + \delta)$, we may observe
\begin{eqnarray}
    \|{\rm grad}_{(m,s^2)}\,\mathfrak{f}((m,s^2),x)\|
    & \leq & \nonumber
    \sqrt{\frac{2}{2\pi}}\,e^{-\frac{(\log(x) - m)^{2}}{2s^{2}}} \Big(\frac{1}{s} +  \frac{|\log(x) - m|}{s^{3}} \Big)\\
    & \leq & \nonumber
    \frac{1}{\sqrt{\pi}}\,\Big(\frac{1}{s} + \frac{|\log(x)| + |m|}{s^{3}} \Big)\\
    & \leq & \label{dritte Ungleichung}
    \frac{1}{\sqrt{\pi}}\,\Big(\frac{1}{\sqrt{s_{0}^{2} - \delta}} + \frac{C + |m_{0}| + \delta}{\sqrt{s_{0}^{2} - \delta}^{3}} \Big).
\end{eqnarray}
By (\ref{erste Ungleichung})--(\ref{dritte Ungleichung}) the function $\mathfrak{h}$ indeed satisfies (\ref{vierte Ungleichung}).

We will next show that $\mathfrak{h}$ is also $\ell$-integrable. For any $\gamma > e^{m_{0} + \delta}$, an application of the change of variable formula yields
\begin{eqnarray*}
    \lefteqn{\int_{(e^{m_{0} + \delta},\gamma)}\frac{1}{\sqrt{\pi}}\,e^{-\frac{(\log(x) - m_{0} -\delta)^{2}}{2(s_{0}^{2} + \delta)}} \Big(\frac{1}{\sqrt{s_{0}^{2} - \delta}} + \frac{|\log(x) - m_{0} + \delta|}{\sqrt{s_{0}^{2} - \delta}^{3}} \Big)\,\ell(dx)}\\
    & = & \int_{m_{0} + \delta}^{\log(\gamma)}\sqrt{\frac{1}{\pi}}\,e^{-\frac{(y - m_{0} -\delta)^{2}}{2(s_{0}^{2} + \delta)}} \Big(\frac{1}{\sqrt{s_{0}^{2} - \delta}} + \frac{| y - m_{0} + \delta|}{\sqrt{s_{0}^{2} - \delta}^{3}} \Big)\,e^{y}\,dy\\
    & = & \int_{0}^{\log(\gamma) - m_{0} - \delta}\frac{1}{\sqrt{\pi}}\,e^{-\frac{z^{2}}{2(s_{0}^{2} + \delta)}} \Big(\frac{1}{\sqrt{s_{0}^{2} - \delta}} + \frac{| z + 2 \delta|}{\sqrt{s_{0}^{2} - \delta}^{3}} \Big)\,e^{z + m_{0} + \delta}\,dz\\
    & = & e^{(2 m_{0} + s_{0}^{2} + \delta)/2}\int_{0}^{\log(\gamma) - m_{0} - \delta}\frac{1}{\sqrt{\pi}}\,e^{-\frac{(z - s_{0}^{2} + \delta)^{2}}{2(s_{0}^{2} + \delta)}} \Big(\frac{1}{\sqrt{s_{0}^{2} - \delta}} + \frac{| z + 2 \delta|}{\sqrt{s_{0}^{2} - \delta}^{3}}\Big)\,dz.
\end{eqnarray*}
Denoting by $Z$ any normally distributed random variable with mean  $s_{0}^{2} - \delta$ and variance $s_{0}^{2} + \delta$ and by $f_Z$ its standard Lebesgue density, we end up with
\begin{eqnarray*}
    \lefteqn{\int_{(e^{m_{0} + \delta},\infty)}|\mathfrak{h}(x)|\,\ell(dx)}\\
    & = & \lim_{\gamma\to\infty} \int_{e^{m_{0} + \delta}}^{\gamma}\frac{1}{\sqrt{\pi}}\,e^{-\frac{(\log(x) - m_{0} -\delta)^{2}}{2(s_{0}^{2} + \delta)}} \Big(\frac{1}{\sqrt{s_{0}^{2} - \delta}} + \frac{|\log(x) - m_{0} + \delta|}{\sqrt{s_{0}^{2} - \delta}^{3}} \Big)\,dx\\
    & \leq & \sqrt{2(s_{0}^{2}+ \delta)}\,e^{(2 m_{0} + s_{0}^{2} + \delta)/2}\int_{0}^{\infty}f_{Z}(z)\Big(\frac{1}{\sqrt{2(s_{0}^{2} - \delta)}} + \frac{| z + 2 \delta|}{\sqrt{s_{0}^{2} -\delta}^{3}} \Big)\,dz\\
    & \leq & \sqrt{2(s_{0}^{2} + \delta)}\,e^{(2 m_{0} + s_{0}^{2} + \delta)/2}\,\ex\Big[\frac{1}{\sqrt{2(s_{0}^{2} - \delta)}} + \frac{| Z + 2 \delta|}{\sqrt{s_{0}^{2} - \delta}^{3}}\Big]\\
    & < & \infty.
\end{eqnarray*}
By definition of $\mathfrak{h}$ this implies that $\mathfrak{h}$ is indeed $\ell$-integrable.

Now, Lemma \ref{Asymptotc Distribution cor - parametric} along with (\ref{Gradientendarstellung}) and (\ref{vierte Ungleichung}) shows that the map $\mathfrak{F}$ is Hadamard differentiable at $(m_0,s_{0}^{2})$ with trace $\bL_1$ and that the Hadamard derivative $\dot \mathfrak{F}_{(m_0,s_0^2)}:\Theta\rightarrow\bL_1$ is given by (\ref{Example Asymptotic Distribution Parametric - Lognormal - derivative}). This proves the first assertion in Example \ref{Example Asymptotic Distribution Parametric - Lognormal}.

For the last assertion in Example \ref{Example Asymptotic Distribution Parametric - Lognormal} we first of all note that $F_{(m_{0},s_{0}^{2})}$ is a continuous function. In particular, it is continuous at $\cR_{\alpha}(F_{(m_{0},s_{0}^{2})})$. It follows by (\ref{def of limit of emp process}) that
\begin{eqnarray}
   \lefteqn{\dot\cR_{\alpha;F_{(m_{0},s_{0}^{2})}}(\dot\mathfrak{F}_{(m_{0},s_{0}^{2})}(\tau_1,\tau_2))}\nonumber\\ 
    &=&
    \frac{(1-\alpha) \int_{(0,{\cal R}_\alpha(F_{(m_{0},s_{0}^{2})})^{+})}\phi_{(0,1)}\big(\frac{\log(x) - m_{0}}{s_{0}}\big)\,\big(\frac{\tau_1}{s_{0}} + \frac{(\log(x) - m_{0})\,\tau_2}{s_{0}^{3}}\big)\,\ell(dx)}{(1-2\alpha)F_{(m_{0},s_{0}^{2})}({\cal R}_\alpha(F_{(m_{0},s_{0}^{2})}))+\alpha}\nonumber\\ 
    & & +\,\frac{\alpha \int_{({\cal R}_\alpha(F_{(m_{0},s_{0}^{2})})^{+},\infty)}\phi_{(0,1)}\big(\frac{\log(x) - m_{0}}{s_{0}}\big)\,\big(\frac{\tau_1}{s_{0}} + \frac{(\log(x) - m_{0})\,\tau_2}{s_{0}^{3}}\big)\,\ell(dx)}{(1-2\alpha)F_{(m_{0},s_{0}^{2})}({\cal R}_\alpha(F_{(m_{0},s_{0}^{2})}))+\alpha}\label{erste Ableitungsformel}
\end{eqnarray}
for all $(\tau_1,\tau_2)\in\Theta$. Let $a := {\cal R}_\alpha(F_{(m_{0},s_{0}^{2})})$. For $b\leq 0$ and any random variable $W$ with distribution function $F_{(m_{0},s_{0}^{2})}$ we have $\alpha\ex[(W-b)^{+}] - (1-\alpha)\ex[(b- W)^{+}] = \alpha (\ex[W] - b) > 0$. Thus $a > 0$ due to (\ref{def expectiles-based shortfall risk measure}). We may apply several times the change of variable formula to obtain
\begin{eqnarray*}
    \lefteqn{\int_{(a,\infty)} \phi_{(0,1)}\Big(\frac{\log(x) - m_{0}}{s_{0}}\Big)\Big(\frac{\tau_1}{s_{0}} + \frac{(\log(x) - m_{0})\tau_2}{s_{0}^{3}}\Big)\,\ell(dx)}\\
    &=&
    \int_{\log(a)}^{\infty} \phi_{(0,1)}\Big(\frac{t - m_{0}}{s_{0}}\Big)\Big(\frac{\tau_1}{s_{0}} + \frac{(t - m_{0})\tau_2}{s_{0}^{3}}\Big)\,e^{t}\,dt\\
    &=&
    \int_{(\log(a) - m_{0})/s_{0}}^{\infty} \phi_{(0,1)}(u)\,\Big(\frac{\tau_1}{s_{0}} + \frac{u\tau_2}{s_{0}^{2}}\Big)\,e^{s_{0} u + m_{0}}\,du\\
    &=&
    e^{m_{0} + s_{0}^{2}/2} \int_{(\log(a) - m_{0})/s_{0}}^{\infty}\frac{e^{-(u - s_{0})^{2}/2}}{\sqrt{2\pi}}\,\Big(\frac{\tau_1}{s_{0}} + \frac{u\tau_2}{s_{0}^{2}}\Big)\,du\\
    &=&
    e^{m_{0} + s_{0}^{2}/2} \int_{(\log(a) - m_{0} - s_{0}^{2})/s_{0}}^{\infty} \frac{e^{-w^{2}/2}}{\sqrt{2\pi}}\,\Big(\frac{\tau_1}{s_{0}} + \frac{(w + s_{0})\tau_2}{s_{0}^{2}}\Big)\,dw\\
    &=& \frac{e^{m_{0} + s_{0}^{2}/2}}{s_{0}}\,\Phi_{(0,1)}\big(\psi(m_{0},s_{0}^{2}, a)\big)\,\tau_1\\
    & & +\,\frac{e^{m_{0} + s_{0}^{2}/2}}{s_{0}^{2}}\Big(s_{0}\Phi_{(0,1)}(\psi(m_{0}, s_{0}^{2},a))+\phi_{(0,1)}(\psi(m_{0},s_{0}^{2},a))\Big)\,\tau_2,
\end{eqnarray*}
where $\psi(m_{0},s_{0}^{2},a) := (m_{0} + s_{0}^{2} - \log(a))/s_{0}$. In the same way we may calculate
\begin{eqnarray*}
    \lefteqn{\int_{(0,a)} \phi_{(0,1)}\Big(\frac{\log(x) - m_{0}}{s_{0}}\Big)\Big(\frac{\tau_1}{s_{0}} + \frac{(\log(x) - m_{0})\tau_2}{s_{0}^{3}}\Big)\,\ell(dx)}\\
    &=&
    \int_{-\infty}^{\log(a)}\phi_{(0,1)}\Big(\frac{t - m_{0}}{s_{0}}\Big)\Big(\frac{\tau_1}{s_{0}} + \frac{(t - m_{0})\tau_2}{s_{0}^{3}}\Big)\,e^{t}\,dt\\
        &=& \frac{e^{m_{0} + s_{0}^{2}/2}}{s_{0}}\,\Phi_{(0,1)}\big(-\psi(m_{0},s_{0}^{2}, a)\big)\,\tau_1\\
    & & +\,\frac{e^{m_{0} + s_{0}^{2}/2}}{s_{0}^{2}}\Big(s_{0}\Phi_{(0,1)}(-\psi(m_{0}, s_{0}^{2},a))+\phi_{(0,1)}(\psi(m_{0},s_{0}^{2},a))\Big)\,\tau_2.
\end{eqnarray*}
Hence in view of (\ref{erste Ableitungsformel}) we end up with
\begin{eqnarray*}
    \lefteqn{\dot\cR_{\alpha;F_{m_{0},s_{0}^{2}}}(\dot\mathfrak{F}_{(m_{0},s_{0}^{2})}(\tau_1,\tau_2))}\\
    &=&
    \frac{e^{m_{0} + s_{0}^{2}/2}}{s_{0}}\,\frac{1 - \alpha - (1 - 2\alpha)\Phi_{(0,1)}\big(\psi(m_{0},s_{0}^{2}, \cR_{\alpha}(F_{(m_{0},s_{0}^{2})}))\big)}{(1-2\alpha)F_{(m_{0},s_{0}^{2})}({\cal R}_\alpha(F_{(m_{0},s_{0}^{2})}))+\alpha}\,(\tau_{1} + \tau_{2})\\
    & & +\,\frac{e^{m_{0} + s_{0}^{2}/2}}{s_{0}^{2}}\,\frac{\phi_{(0,1)}\big(\psi(m_{0},s_{0}^{2},\cR_{\alpha}(F_{(m_{0},s_{0}^{2})}))\big)}{(1-2\alpha)F_{(m_{0},s_{0}^{2})}({\cal R}_\alpha(F_{(m_{0},s_{0}^{2})}))+\alpha}\, \tau_{2}.
\end{eqnarray*}
Therefore, $\dot\cR_{\alpha,F_{(m_{0},s_{0}^{2})}}(\dot\mathfrak{F}_{(m_{0},s_{0}^{2})}(Y_{(m_{0},s_{0}^{2}}))$ is centered normal with variance as in (\ref{Example Asymptotic Distribution Parametric - Lognormal - variance}).
\hfill$\Box$


\subsection{Proof of Example \ref{Example Asymptotic Distribution Parametric - Pareto}}\label{Proof of Example Asymptotic Distribution Parametric - Pareto}

For $b\leq \overline{c}$ and any random variable $W$ with distribution function $F_{a_{0}}$ we have
$$
\alpha\ex[(W-b)^{+}] - (1-\alpha)\ex[(b- W)^{+}] = \alpha (\ex[W] - b) = \frac{\overline{c}}{a - 1} > 0.
$$
Thus $\cR_{\alpha}(F_{a_{0}}) > \overline{c}$ for every $a_{0}\in\Theta$ due to (\ref{def expectiles-based shortfall risk measure}). Now, invoking (\ref{def of limit of emp process}), we may draw on (\ref{Example Asymptotic Distribution Parametric - Pareto - derivative}) to observe that for any $a_{0}\in\Theta$
\begin{eqnarray*}
    \lefteqn{\dot\cR_{\alpha;F_{a_{0}}}(\dot\mathfrak{F}_{a_{0}}(Y_{a_{0}}))}\\
    & = & - Y_{a_{0}}\,\frac{(1 - \alpha)\int_{\overline{c}}^{\cR_{\alpha}(F_{a_{0}})}\log(\overline{c}/x)\,(\overline{c}/x)^{a_{0}}\,dx\,+\,
    \alpha\int_{\cR_{\alpha}(F_{a_{0}})}^{\infty}\log(\overline{c}/x)\,(\overline{c}/x)^{a_{0}}\,dx}
{(1 - 2\alpha) F_{a_{0}}(\cR_{\alpha}(F_{a_{0}}))+ \alpha}\,.
\end{eqnarray*}
Routine calculations yield
\begin{eqnarray*}
    \int_{\beta}^{\gamma}\log(\overline{c}/x)\,(\overline{c}/x)^{a_{0}}\,dx
    & = & \frac{\overline{c}}{(1- a_{0})^{2}}\,(\gamma/\overline{c})^{1 - a_{0}} (1 - (1- a_{0})\log(\gamma/\overline{c}))\\
    & &-\frac{\overline{c}}{(1- a_{0})^{2}}\,(\beta/\overline{c})^{1 - a_{0}} (1 - (1- a_{0})\log(\beta/\overline{c}))
\end{eqnarray*}
for $a_{0}\in\Theta$ and $\overline{c}\leq\beta < \gamma < \infty$. Hence for every $a_{0}\in\Theta$ the random variable $\dot\cR_{\alpha;F_{a_{0}}}(\dot\mathfrak{F}_{a_{0}}(Y_{a_{0}}))$ is centered normal with variance given by (\ref{Example Asymptotic Distribution Parametric - Pareto - 10}).
\hfill$\Box$


\appendix


\section{Expectiles as risk measures on $L^1$}\label{expectiles as risk measures}

Let $(\Omega,{\cal F},\pr)$ be an atomless probability space and use $L^1=L^1(\Omega,{\cal F},\pr)$ to denote the usual $L^1$-space. Pick $\alpha\in(0,1)$ and let $\U_\alpha$ be as in (\ref{Def Mapping bb U alpha}).

\begin{lemma}\label{representation int by parts - mathbb}
For every $X\in L^1$, the mapping $m\mapsto\mathbb{U}_{\alpha}(X)(m)$ is real-valued, continuous, strictly decreasing, and satisfies $\lim_{m\to\pm\infty}\mathbb{U}_\alpha(X)(m)=\mp\infty$. In addition it may be represented by
\begin{eqnarray}
    \mathbb{U}_\alpha(X)(m)
    & = & -(1-\alpha)\int_{(-\infty,m)} F_X(x)\,\ell(dx)\,+\,\alpha\int_{(m,\infty)} (1-F_X(x))\,\ell(dx)\label{representation int by parts - mathbb - eq - 1}\\
    & = & -(1-\alpha)\int_{(-\infty,0)} F_X(x+m)\,\ell(dx)\,+\,\alpha\int_{(0,\infty)} (1-F_X(x+m))\,\ell(dx),\qquad\label{representation int by parts - mathbb - eq - 2}
\end{eqnarray}
where $F_X$ is the distribution function of $X$.
\end{lemma}

\begin{proof}
First of all note that the expectation $\ex[U_\alpha(X-m)]$ exists for every $m\in\R$, because $X\in L^1$. Further, we have
\begin{eqnarray*}
    \lefteqn{\mathbb{U}_\alpha(X)(m)}\\
    & = &
    \alpha\,\ex[(X - m)^{+}]\,-\, (1 - \alpha)\,\ex[(- X - (-m) )^{+} ]\\
    & = & \alpha\int_{(m,\infty)}(1 - F_X(x))\,\ell(dx)\,-\,(1- \alpha)\int_{(-m,\infty)} (1 - F_{-X}(x))\,\ell(dx)\\
    & = & \alpha\int_{(m,\infty)}(1 - F_X(x))\,\ell(dx)\,-\, (1- \alpha)\int_{(-m,\infty)} F_X(-x)\,\ell(dx)\quad\mbox{for all }m\in\R.
\end{eqnarray*}
That is, (\ref{representation int by parts - mathbb - eq - 1}) holds. Applying Change of Variable to both integrals yields (\ref{representation int by parts - mathbb - eq - 2}).

For every $m\in\R$ and $\varepsilon > 0$, we obtain by (\ref{representation int by parts - mathbb - eq - 1}) that
$$
    \mathbb{U}_{\alpha}(X)(m + \varepsilon) - \mathbb{U}_{\alpha}(X)(m)\,=\,-\int_{(m,m + \varepsilon)} \Big((1 - \alpha)F_X(X) + \alpha (1 - F_X(x))\Big)\,\ell(dx)\,<\,0.
$$
Thus, $m\mapsto\mathbb{U}_{\alpha}(X)(m)$ is strictly decreasing. By the dominated convergence theorem we also have
$$
    \lim_{\widetilde{m}\to m}\ex[(X - \widetilde{m})^{+}] = \ex[(X - m)^{+}]\quad\mbox{ and }\quad\lim_{\widetilde{m}\to m}\ex[(\widetilde{m} - X)^{+}] = \ex[(m - X)^{+}]
$$
for every $m\in\R$, implying continuity of $m\mapsto\mathbb{U}_{\alpha}(X)(m)$. By the Monotone Convergence theorem, we have
$$
    \lim_{m\to -\infty}\ex[(X - m)^{+}]\,=\,\infty\quad\mbox{ and }\quad\lim_{m\to\infty}\ex[(m - X)^{+}]\,=\,\infty,
$$
and by the dominated convergence theorem we obtain
$$
    \lim_{m\to \infty}\ex[(X - m)^{+}]\,=\,0\quad\mbox{ and }\quad\lim_{m\to-\infty}\ex[(m - X)^{+}]\,=\,0.
$$
This gives
$$
    \lim_{m\to -\infty}\mathbb{U}_{\alpha}(X)(m)\,=\,\infty\quad\mbox{ and }\quad\lim_{m\to \infty}\mathbb{U}_{\alpha}(X)(m)\,=\,-\infty.
$$
The proof is complete.
\end{proof}

Lemma \ref{representation int by parts - mathbb} ensures that (\ref{def expectiles-based shortfall risk measure}) defines a map $\rho_\alpha:L^1\to\R$. The following proposition shows that this map is a coherent risk measure when $1/2\le\alpha<1$.

\begin{proposition}\label{Expectiles as coherent RM}
The map $\rho_{\alpha}:L^1\to\R$ is monotone, cash-invariant, positively homogeneous, and continuous w.r.t.\ the $L^{1}$-norm $\|\cdot\|_{1}$. It is subadditive (and thus coherent) if and only if $1/2\le\alpha<1$. If $0<\alpha<1/2$, then the map $\check\rho_{\alpha}:L^1\to\R$ defined by $\check\rho_\alpha(X):=-\rho_\alpha(-X)$ provides a $\|\cdot\|_1$-continuous coherent risk measure.
\end{proposition}

\begin{proof}
In view of Lemma \ref{representation int by parts - mathbb} it may be verified easily that the map $\rho_\alpha$ is monotone, cash-invariant, and positively homogeneous. Concerning subadditivity for $1/2\le\alpha<1$, we want to show that $\rho_{\alpha}$ is a convex mapping. For this purposes let $X_{1}, X_{2}\in L^{1}$ and $\lambda\in [0,1]$. By convexity of $U_{\alpha}$ we may observe
\begin{eqnarray*}
    0
    & \ge &
    \lambda\mathbb{U}_{\alpha}(X_1)(\rho_{\alpha}(X_1)) + (1-\lambda)\mathbb{U}_\alpha(X_2)(\rho_{\alpha}(X_2))\\
    & = & \lambda\ex[{U}_{\alpha}(X_1-\rho_{\alpha}(X_1))] + (1-\lambda)\ex[U_\alpha(X_2-\rho_{\alpha}(X_2))]\\
    & \ge & \ex[{U}_{\alpha}(\lambda(X_1-\rho_{\alpha}(X_1))+ (1-\lambda)(X_2-\rho_{\alpha}(X_2)))]\\
    & = & \mathbb{U}_{\alpha}(\lambda X_1 + (1-\lambda)X_2)(\lambda\rho_{\alpha}(X_1) + (1-\lambda)\rho_{\alpha}(X_2)).
\end{eqnarray*}
Since by Lemma \ref{representation int by parts - mathbb} the mapping $m\mapsto\mathbb{U}_{\alpha}(X)(m)$ is strictly decreasing for any $X\in L^{1}$, we may conclude
$$
    \lambda\rho_{\alpha}(X_1) + (1-\lambda)\rho_{\alpha}(X_2)\ge\rho_{\alpha}(\lambda X_1 + (1-\lambda)X_2).
$$
This shows convexity of $\rho_{\alpha}$ which along with positive homogeneity implies subadditivity. Hence $\rho_{\alpha}$ is a coherent risk measure on $L^{1}$. so that it is also continuous w.r.t. $\|\cdot\|_{1}$ due to Theorem 4.1 in \cite{CheriditoLi2009}. Moreover, by Proposition 6 in \cite{Bellinietal2014}, the restriction of $\rho_{\alpha}$ to $L^{2}$ is a coherent risk measure if and only if $1/2\le\alpha<1$. This proves the first part of Proposition \ref{Expectiles as coherent RM}, except the $\|\cdot\|_1$-continuity for $0<\alpha<1/2$.

To prove the second part, let $0<\alpha<1/2$. In this case the mapping $m\mapsto U_{\alpha}(m)$ is concave, which implies
\begin{eqnarray*}
    0
    & \le &
    \lambda\mathbb{U}_{\alpha}(X_1)(\rho_{\alpha}(X_1)) + (1-\lambda)\mathbb{U}_\alpha(X_2)(\rho_{\alpha}(X_2))\\
    & = & \lambda\ex[{U}_{\alpha}(X_1-\rho_{\alpha}(X_1))] + (1-\lambda)\ex[U_\alpha(X_2-\rho_{\alpha}(X_2))]\\
    & \le & \ex[{U}_{\alpha}(\lambda(X_1-\rho_{\alpha}(X_1))+ (1-\lambda)(X_2-\rho_{\alpha}(X_2)))]\\
    & = & \mathbb{U}_{\alpha}(\lambda X_1 + (1-\lambda)X_2)(\lambda\rho_{\alpha}(X_1) + (1-\lambda)\rho_{\alpha}(X_2))
\end{eqnarray*}
for any $X_1,X_2\in L^{1}$ and $\lambda\in[0,1]$. Again by Lemma \ref{representation int by parts - mathbb} the mapping $m\mapsto\mathbb{U}_{\alpha}(X)(m)$ is strictly decreasing for any $X\in L^{1}$. We thus obtain
$$
    \lambda\rho_{\alpha}(X_1) + (1-\lambda)\rho_{\alpha}(X_2)\le\rho_{\alpha}(\lambda X_1 + (1-\lambda)X_2)\quad\mbox{ for all }X_1,X_2\in L^{1},\,\lambda\in [0,1].
$$
Therefore the map $\check{\rho}_\alpha:L^1\rightarrow\R$ defined by $\check\rho_\alpha(X):=-\rho_\alpha(-X)$ is convex. It is also monotone, cash-invariant, positively homogeneous, and thus a coherent risk measure; note again that subadditivity follows from convexity under the condition of positive homogeneity. As a coherent risk measure on $L^{1}$, the map $\check\rho_{\alpha}$ is $\|\cdot\|_1$-continuous, drawing again on \cite[Theorem 4.1]{CheriditoLi2009}. Then, obviously, $\rho_{\alpha}$ is $\|\cdot\|_1$-continuous too which completes the proof.
\end{proof}


\section{Quasi-Hadamard differentiability and functional delta-method}\label{appendix QHD and FDM}

Let $\V$ and $\widetilde\V$ be vector spaces, and let $\E\subseteq\V$ and $\widetilde\E\subseteq\widetilde\V$ be subspaces equipped with norms $\|\cdot\|_{\E}$ and $\|\cdot\|_{\widetilde\E}$, respectively.

\begin{definition}\label{definition quasi hadamard}
Let $H:\V_H\rightarrow\widetilde\V$ be a map defined on a subset $\V_H\subseteq\V$, and $\E_0$ be a subset of $\E$. Then $H$ is said to be quasi-Hadamard differentiable at $x\in \V_H$ tangentially to $\E_0\langle\E\rangle$ with trace $\widetilde\bE$ if $H(x)-H(y)\in\widetilde\bE$ for all $y\in\V_H$ and there exists a continuous map $\dot H_x:\E_0\to\widetilde\E$ such that
\begin{eqnarray}\label{def eq for HD}
    \lim_{n\to\infty}\Big\|\dot H_x(x_0)-\frac{H(x+\varepsilon_nx_n)-H(x)}{\varepsilon_n}\Big\|_{\widetilde\bE}\,=\,0
\end{eqnarray}
holds for each triplet $(x_0,(x_n),(\varepsilon_n))$, with $x_0\in\E_0$, $(\varepsilon_n)\subseteq(0,\infty)$ satisfying $\varepsilon_n\to 0$, $(x_n)\subseteq\E$ satisfying $\|x_n-x_0\|_{\E}\to 0$ as well as $(x+\varepsilon_nx_n)\subseteq\V_H$. In this case the map $\dot H_x$ is called quasi-Hadamard derivative of $H$ at $x$ tangentially to $\E_0\langle\E\rangle$ with trace $\widetilde\bE$.
\end{definition}

\begin{remarknorm}\label{definition quasi hadamard - remark}
(i) When $\widetilde\V=\widetilde\bE$, then $H(x)-H(y)\in\widetilde\bE$ automatically holds for all $x,y\in\V_H$ and the notion of quasi-Hadamard differentiability of $H$ at $x\in \V_H$ tangentially to $\E_0\langle\E\rangle$ with trace $\widetilde\bE$ coincides with the notion of quasi-Hadamard differentiability of $H$ at $x$ tangentially to $\E_0\langle\E\rangle$ as introduced in \cite{BeutnerZaehle2010,BeutnerZaehle2015}.

(ii) When $\widetilde\V=\widetilde\bE$, $\E_0=\E$, and $\|\cdot\|_{\E}$ provides a norm on all of $\V$, then the notion of quasi-Hadamard differentiability of $H$ at $x$ tangentially to $\E_0\langle\E\rangle$ coincides with the classical notion of Hadamard differentiability at $x$ tangentially to $\E$ as defined in \cite{Gill1989}. However, in general the Hadamard derivative of $H$ at $x$ tangentially to $\E$ is not necessarily the same as the quasi-Hadamard derivative of $H$ tangentially to $\E\langle\E\rangle$, because in the latter case the norm $\|\cdot\|_{\E}$ may be defined only on $\E$ (and not on all of $\V$).

(iii) When $\E_0=\E$ and $\|\cdot\|_{\E}$ provides a norm on all of $\V$ then we skip the prefix ``quasi'' and speak of Hadamard differentiability of $H$ at $x$ tangentially to $\bE$ with trace $\widetilde\bE$, and when even $\E_0=\E=\V$ then we in addition skip the suffix ``tangentially to $\bE$''.
{\hspace*{\fill}$\Diamond$\par\bigskip}
\end{remarknorm}

The discussion in part (ii) of the preceding remark shows in particular that quasi-Hadamard differentiability is a weaker notion of ``differentiability'' than the classical (tangential) Hadamard differentiability. However, Theorem \ref{delta method for the bootstrap} shows that this notion is still strong enough to obtain a functional delta-method (even for the bootstrap).

Denote by ${\cal B}^\circ$ the $\sigma$-algebra on $\bE$ that is generated by the open balls. Convergence in distribution in $\bE$ will be considered for the open-ball $\sigma$-algebra. More precisely, let $(\xi_n)$ be a sequence of $(\bE,{\cal B}^\circ)$-valued random variables on a probability space $(\Omega',{\cal F}',\pr')$, and $\xi$ be an $(\bE,{\cal B}^\circ)$-valued random variable on some probability space $(\check\Omega,\check{\cal F},\check\pr)$. Then $(\xi_n)$ is said to converge in distribution to $\xi$, in symbols $\xi_n\leadsto^\circ\xi$, if $\int f\,d\pr_{\xi_n}'\to\int f\,d\check\pr_{\xi}$ for all bounded, continuous and $({\cal B}^\circ,{\cal B}(\R))$-measurable functions $f:\bE\rightarrow\R$. Note that, whenever $\xi$ concentrates on a separable measurable set, we have $\xi_n\leadsto^\circ\xi$ if and only if $\varrho_{\scriptsize{\rm BL}}^\circ(\pr_{\xi_n},\check\pr_{\xi})\to 0$, where $\varrho_{\scriptsize{\rm BL}}^\circ$ is the bounded Lipschitz distance defined by
$$
    \varrho_{\scriptsize{\rm BL}}^\circ(\mu,\nu)\,:=\,\sup_{f\in{\rm BL}_1^\circ}\Big|\int f\,d\mu-\int f\,d\nu\Big|
$$
with ${\rm BL}_1^\circ$ the set of all $({\cal B}^\circ,{\cal B}(\R))$-measurable $f:\bE\rightarrow\R$ satisfying $|f(x)-f(y)|\le\|x-y\|_{\bE}$ for all $x,y\in\bE$ and $\sup_{x\in\bE}|f(x)|\le 1$. If $(\bE,\|\cdot\|_{\bE})$ is separable, then ${\cal B}^\circ$ coincides with the Borel $\sigma$-algebra ${\cal B}$ on $\bE$. Then the notion of convergence $\xi_n\leadsto^\circ\xi$ boils down to the conventional notion of convergence in distribution, because every continuous function $f:\bE\rightarrow\R$ is $({\cal B},{\cal B}(\R))$-measurable. We then also write $\xi_n\leadsto\xi$ and $\varrho_{\scriptsize{\rm BL}}$ instead of $\xi_n\leadsto^\circ\xi$ and $\varrho_{\scriptsize{\rm BL}}^\circ$, respectively.

Now, let $(\Omega,{\cal F},\pr)$ be a probability space, and $(\widehat T_n)$ be a sequence of maps $\widehat T_n:\Omega\rightarrow\V$. Regard $\omega\in\Omega$ as a sample drawn from $\pr$, and $\widehat T_n(\omega)$ as a statistic derived from $\omega$. Let $\theta\in\V$, and $(a_n)$ be a sequence of positive real numbers tending to $\infty$. Let $(\Omega',{\cal F}',\pr')$ be another probability space and set $(\overline\Omega,\overline{\cal F},\overline{\pr}):=(\Omega\times\Omega',{\cal F}\otimes{\cal F}',\pr\otimes\pr')$. The probability measure $\pr'$ represents a random experiment that is run independently of the random sample mechanism $\pr$. Below, $\widehat T_n$ will also be regarded as a map defined on the extension $\overline\Omega$ of $\Omega$. Let $(\widehat T_n^*)$ be a sequence of maps $\widehat T_n^*: \overline\Omega\rightarrow\V$. Finally denote by $\widetilde{\cal B}$ and $\widetilde\varrho_{\scriptsize{\rm BL}}$ the Borel $\sigma$-algebra on $\widetilde\bE$ and the bounded Lipschitz distance on $\widetilde\bE$, respectively. The following theorem is a slight generalization of Theorem 3.1 in \cite{BeutnerZaehle2015}; one can use the same proof with the obvious (minor) modifications.

\begin{theorem}\label{delta method for the bootstrap}
Let $H:\V_H\to\widetilde\E$ be a map defined on a subset $\V_H\subseteq\V$. Let $\bE_0\subseteq\bE$ be a separable subspace and assume that $\bE_0\in {\cal B}^{\circ}$. Assume that $(\widetilde\bE,\|\cdot\|_{\widetilde\bE})$ is separable, let $(a_n)$ be a sequence of positive real numbers tending to $\infty$, and consider the following conditions:
\begin{itemize}
    \item[(a)] $a_n(\widehat T_n-\theta)$ takes values only in $\bE$, is $({\cal F},{\cal B}^{\circ})$-measurable, and satisfies
    $$
         a_n(\widehat T_n-\theta)\,\leadsto^\circ\,\xi\qquad\mbox{in $(\bE,{\cal B}^{\circ},\|\cdot\|_{\bE})$}
    $$
    for some $(\bE,{\cal B}^{\circ})$-valued random variable $\xi$ on some probability space $(\check\Omega,\check{\cal F},\check\pr)$ with $\xi(\check\Omega)\subseteq\bE_0$.
    \item[(b)] $a_n(H(\widehat T_n)-H(\theta))$ takes values only in $\widetilde\bE$ and is $({\cal F},\widetilde{\cal B})$-measurable.
    \item[(c)] The map $H$ is quasi-Hadamard differentiable at $\theta$ tangentially to $\bE_0\langle\bE\rangle$ with trace $\widetilde\bE$ and quasi-Hadamard derivative $\dot H_\theta$.
    \item[(d)] The quasi-Hadamard derivative $\dot H_\theta$ can be extended from $\bE_0$ to $\bE$ such that the extension $\dot H_\theta:\bE\rightarrow\widetilde\bE$ is linear and $({\cal B}^{\circ},\widetilde{\cal B})$-measurable. Moreover, the extension $\dot H_\theta:\bE\rightarrow\widetilde\bE$ is continuous at every point of $\bE_0$.
    \item[(e)] $a_n(H(\widehat T_{n}^*)-H(\widehat T_{n}))$ takes values only in $\widetilde\bE$ and is $(\overline{\cal F},\widetilde{\cal B})$-measurable.
    \item[(f)] $a_n(\widehat T_n^*-\theta)$ and $a_n(\widehat T_n^{*}-\widehat T_n)$ take values only in $\bE$ and are $(\overline{\cal F},{\cal B}^\circ)$-measurable, and $$
            a_n(\widehat T_n^{*}(\omega,\cdot)-\widehat T_n(\omega))\,\leadsto^\circ\,\xi\qquad\mbox{in $(\bE,\mathcal{B}^{\circ},\|\cdot\|_{\bE})$},\qquad\mbox{$\pr$-a.e.\ $\omega$}.
        $$
    \item[(f')] $a_n(\widehat T_n^*-\theta)$ and $a_n(\widehat T_n^{*}-\widehat T_n)$ take values only in $\bE$ and are $(\overline{\cal F},{\cal B}^\circ)$-measurable, and
        $$
            \lim_{n\to\infty}\pr^{\scriptsize{\sf out}}\big[\big\{\omega\in\Omega:\,\varrho_{\scriptsize{\rm BL}}^\circ(\pr'_{a_n(\widehat T_n^{*}(\omega,\cdot)-\widehat T_n(\omega))},\check\pr_{\xi})\ge\delta\big\}\big]=\,0\quad\mbox{ for all }\delta>0.
        $$
\end{itemize}
Then the following assertions hold:
\begin{itemize}
    \item[(i)] If conditions (a)--(c) hold, then $\dot H_\theta(\xi)$ is $(\check{\cal F},{\cal B}(\R))$-measurable and
        $$
            a_n(H(\widehat T_n)-H(\theta))\,\leadsto\,\dot H_\theta(\xi)\qquad\mbox{in $(\widetilde\bE,\widetilde{\cal B},\|\cdot\|_{\widetilde\bE)}$}.
        $$
    \item[(ii)] If conditions (a)--(f) hold, then $\dot H_\theta(\xi)$ is $(\check{\cal F},{\cal B}(\R))$-measurable and
        $$
            \lim_{n\to\infty}\pr\big[\big\{\omega\in\Omega:\,\widetilde\varrho_{\scriptsize{\rm BL}}\big(\pr'_{a_n(H(\widehat T_n^{*}(\omega,\cdot))-H(\widehat T_n(\omega)))},\check\pr_{\dot H_\theta(\xi)}\big)\ge\delta\big\}\big]=\,0\quad\mbox{ for all }\delta>0.
        $$
    \item[(iii)]  Assertion (ii) still holds when assumption (f) is replaced by (f').
\end{itemize}
\end{theorem}

For (f) and (f') in the preceding theorem note that the mapping $\omega'\mapsto a_n(\widehat T_n^*(\omega,\omega')-\widehat T_n(\omega))$ is $({\cal F}',{\cal B}^\circ)$-measurable for every fixed $\omega\in\Omega$, because $a_n(\widehat T_n^*-\widehat T_n)$ is $(\overline{\cal F},{\cal B}^\circ)$-measurable with $\overline{\cal F}={\cal F}\otimes{\cal F}'$. That is, $a_n(\widehat T_n^*(\omega,\cdot)-\widehat T_n(\omega))$ can be seen as an $(\bE,{\cal B}^\circ)$-valued random variable on $(\Omega',{\cal F}',\pr')$ for every fixed $\omega\in\Omega$. Analogously, we can regard $a_n(H(\widehat T_n^{*}(\omega,\cdot))-H(\widehat T_n(\omega)))$ as a real-valued random variable on $(\Omega',{\cal F}',\pr')$ for every fixed $\omega\in\Omega$. This matters for the formulation of part (ii) in the preceding theorem.


\section{Convergence in distribution of the empirical process regarded as an $\bL_1$-valued random variable}\label{some comments on the space L 1}

By definition $\bL_1$ is the set of all Borel measurable functions $v:\R\rightarrow\R$ with $\|v\|_{1,\ell}<\infty$ modulo the equivalence relation of almost sure identity, where $\|\cdot\|_{1,\ell}$ is defined in (\ref{Def Wasserstein norm}). It is known that $(\bL_1,\|\cdot\|_{1,\ell})$ is a separable Banach space; cf.\ Theorem 4.1.3 and Corollary 4.2.2 in \cite{Bogachev2007}. Denote by ${\cal B}_1$ the Borel $\sigma$-algebra on $\bL_1$ w.r.t.\ the norm $\|\cdot\|_{1,\ell}$. Let $\xi$ be a real-valued stochastic process on a probability space $(\Omega,{\cal F},\pr)$ with index set $\R$. That is, $\xi:\Omega\times\R\rightarrow\R$ is any map such that the coordinate $\omega\mapsto\xi(t,\omega)$ is $({\cal F},{\cal B}(\R))$-measurable for every $t\in\R$. The process $\xi$ is said to be measurable if $\xi:\Omega\times\R\rightarrow\R$ is $({\cal F}\otimes{\cal B}(\R),{\cal B}(\R))$-measurable.

\begin{lemma}\label{L 1 valued random variables}
If the stochastic process $\xi$ is measurable and $\xi(\omega,\cdot)\in\bL_1$ for all $\omega\in\Omega$, then $\omega\mapsto\xi(\omega,\cdot)$ is an $({\cal F},{\cal B}_1)$-measurable mapping from $\Omega$ to $\bL_1$. In particular, $\xi$ can be seen as an $(\bL_1,{\cal B}_1)$-valued random variable.
\end{lemma}

\begin{proof}
The family of all open balls generate ${\cal B}_1$, because $(\bL_1,{\cal B}_1)$ is separable. Thus it suffices to show that $\xi^{-1}(B_r(v))\in{\cal F}$ for every $r>0$ and $v\in\bL_1$, where $B_r(v)$ denotes the $\|\cdot\|_{1,\ell}$-open ball with radius $r>0$ around $v\in\bL_1$. Let $r>0$ and $v\in\bL_1$. By assumption the map $\xi:\bL_1\times\R\rightarrow\R$ is $({\cal F}\otimes{\cal B}(\R),{\cal B}(\R))$-measurable. Since the mapping $t\mapsto v(t)$ is Borel measurable for every $v\in\bL_1$, it follows that also the map $\xi_v:\bL_1\times\R\rightarrow\R$ defined by $\xi_v(\omega,t):=v(t)$ is $({\cal F}\otimes{\cal B}(\R),{\cal B}(\R))$-measurable. By Fubinis' theorem we obtain in particular that the map ${\cal I}_v:\Omega\rightarrow\R$ defined by ${\cal I}_v(\omega):=\int|\xi_v(\omega,t)-\xi(\omega,t)|\,\ell(dt)$ is $({\cal F},{\cal B}(\R))$-measurable. Along with
$$
    \xi^{-1}(B_r(v))=\Big\{\omega\in\Omega:\,\int|\xi_v(\omega,t)-\xi(\omega,t)|\,\ell(dt)<r\Big\}={\cal I}_v^{-1}([0,r)),
$$
this implies $\xi^{-1}(B_r(v))\in{\cal F}$.
\end{proof}

\begin{remarknorm}
It is well known that every real-valued stochastic process $\xi$ with right-continuous paths is measurable. In particular, the process $\sqrt{n}(F_{\widehat\theta_n}-F)$ is measurable when $F$ is a distribution function and $F_{\widehat\theta_n}$ is a process with right-continuous paths. It follows that $\sqrt{n}(F_{\widehat\theta_n}-F)$ can be seen as an $\bL_1$-valued random variable when $F\in\F_1$ and $F_{\widehat\theta_n}$ takes values only in $\F_1$.
{\hspace*{\fill}$\Diamond$\par\bigskip}
\end{remarknorm}

The following Theorem \ref{CLT by Dede} recalls the statements of Propositions 3.2 and 3.5 in \cite{Dede2009}. Here $\bL_\infty$ refers to the space of all bounded Borel measurable functions from $\R$ to $\R$ modulo the equivalence relation of $\ell$-almost sure identity.

\begin{theorem}\label{CLT by Dede}
With the notation and under the assumptions of Theorem \ref{main theorem coroll} (except the continuity of $F$ at ${\cal R}_\alpha(F)$),
$$
    \sqrt{n}(\widehat F_n-F)\,\leadsto\,B_F\qquad\mbox{in $(\bL_1,{\cal B}_1,\|\cdot\|_{1,\ell})$}
$$
for an $\bL_1$-valued centered Gaussian random variable $B_F$ with covariance operator
\begin{equation}\label{CLT by Dede - eq}
    \Phi_{B_F}(f,g)=\int_{\R^2}f(s)C_F(s,t)g(t)\,d(s,t)\qquad\mbox{for all }~f,g\in\bL_\infty,
\end{equation}
where $C_F(s,t)$ is defined by (\ref{main theorem coroll - eq - 60}).
\end{theorem}

Recall from \cite{AraujoGine1980} that an $(\bL_1,{\cal B}_1)$-valued random variable $B$ on some probability space $(\Omega,{\cal F},\pr)$ is said to be an $\bL_1$-valued Gaussian random variable if $\Lambda(B)$ is a real-valued Gaussian random variable for each $\|\cdot\|_{1,\ell}\,$-continuous linear functional $\Lambda:\bL_1\rightarrow\R$, i.e, if $\int f(t)B(t)\,\ell(dt)$ is a real-valued Gaussian random variable for every $f\in\bL_\infty$. The covariance operator of such an $\bL_1$-valued Gaussian random variable $B$ is the mapping $\Phi_{B}:\bL_\infty\times\bL_\infty\rightarrow\R$ defined by
$$
    \Phi_{B}(f,g):=\ex\Big[\Big(\int f(s)(B(s)-\ex[B(s)])\,\ell(ds)\Big)\Big( \int g(t)(B(t)-\ex[B(t)])\,\ell(dt)\Big)\Big].
$$

\begin{theorem}\label{bootstrap results of VdV-W}
With the notation and under the assumptions of Theorem \ref{main theorem coroll - bootstrap} (except the continuity of $F$ at ${\cal R}_\alpha(F)$),
\begin{equation}\label{bootstrap results of VdV-W - 20 - new}
    \sqrt{n}(\widehat F_n^{*}(\omega,\cdot)-\widehat F_n(\omega))\,\leadsto\,B_F\qquad\mbox{in $(\bL_1,\mathcal{B}_1,\|\cdot\|_{1,\ell})$},\qquad\mbox{$\pr$-a.e.\ $\omega$},
\end{equation}
where $B_F$ is as in Theorem \ref{CLT by Dede} (with $C_F(t_0,t_1)=F(t_0\wedge t_1)(1-F(t_0\vee t_1))$).
\end{theorem}

\begin{proof}
Theorem 5.2 in \cite{BeutnerZaehle2015} shows that the imposed assumptions imply that (\ref{bootstrap results of VdV-W - 20 - new}) with $\leadsto$ and $(\bL_1,\mathcal{B}_1,\|\cdot\|_{1,\ell})$ replaced by $\leadsto^\circ$ and $(\D_\phi,{\cal D}_\phi,\|\cdot\|_\phi)$ holds. Here $\D_\phi$ is the space of all \cadlag\ functions $v:\R\rightarrow\R$ with $\|v\|_\phi:=\sup_{x\in\R}|v(x)|\phi(x)<\infty$ and ${\cal D}_\phi$ is the open-ball $\sigma$-algebra on $(\D_\phi,\|\cdot\|_\phi)$. Since $\D_\phi\subseteq\bL_1$ and $\|\cdot\|_{1,\ell}\le C_\phi\|\cdot\|$ with $C_\phi:=\int 1/\phi\,d\ell<\infty$, the natural embedding $\D_\phi\rightarrow\bL_1$, $v\mapsto v$, is $(\|\cdot\|_\phi,\|\cdot\|_{1,\ell})$-continuous. Thus the continuous mapping theorem in the form of \cite[Theorem 6.4]{Billingsley1999} ensures that (\ref{bootstrap results of VdV-W - 20 - new}) itself holds too.
\end{proof}

\begin{theorem}\label{bootstrap results of Radulovic EP}
With the notation and under the assumptions of Theorem \ref{bootstrap results of Radulovic} (except the continuity of $F$ at ${\cal R}_\alpha(F)$),
\begin{equation}\label{bootstrap results of Radulovic EP - 10}
    \sqrt{n}(\widehat F_n^{*}(\omega,\cdot)-\widehat F_n(\omega))\,\leadsto\,B_F\qquad\mbox{in $(\bL_1,\mathcal{B}_1,\|\cdot\|_{1,\ell})$},\qquad\mbox{$\pr$-a.e.\ $\omega$},
\end{equation}
where $B_F$ is as in Theorem \ref{CLT by Dede}.
\end{theorem}

\begin{proof}
One can argue as in the proof of Theorem \ref{bootstrap results of VdV-W} (with Theorem 5.4 of \cite{BeutnerZaehle2015} in place of Theorem 5.2 in  \cite{BeutnerZaehle2015}).
\end{proof}


\section{A note on the dependence coefficients $\widetilde\phi$ and $\widetilde\alpha$}\label{dependence coefficients}

Let $(X_i)_{i\in\N}$ be a strictly stationary and ergodic sequence of real-valued random variables on some probability space $(\Omega,{\cal F},\pr)$. At the beginning of Section \ref{section asymptotic distribution} we claimed that Dedecker and Prieur \cite{DedeckerPrieur2005} introduced the following dependence coefficients:
\begin{eqnarray*}
    \widetilde\phi(n) & := & \sup_{k\in\N}\,\sup_{x\in\R}\,\|\,\pr[X_{n+k}\in(-\infty,x]|{\cal F}_1^k](\cdot)-\pr[X_{n+k}\in(-\infty,x]]\,\|_\infty,\\
    \widetilde\alpha(n) & := & \sup_{k\in\N}\,\sup_{x\in\R}\,\|\,\pr[X_{n+k}\in (-\infty,x]|{\cal F}_1^k](\cdot)-\pr[X_{n+k}\in(-\infty,x]]\,\|_1,
\end{eqnarray*}
where ${\cal F}_1^k:=\sigma(X_1,\ldots,X_k)$ and $\|\cdot\|_p$ denotes the usual $L^p$-norm on $L^p=L^p(\Omega,{\cal F},\pr)$, $p\in[1,\infty]$. Note, however, that in \cite{DedeckerPrieur2005} the starting point is a strictly stationary and ergodic sequence $(Y_{i})_{i\in\Z}$, and the above dependence coefficients are actually defined by
\begin{eqnarray*}
    \overline{\phi}(n) & := & \sup_{x\in\R}\,\|\,\pr[Y_{n}\in(-\infty,x]|{\cal F}^{0}](\cdot)-\pr[Y_{n}\in(-\infty,x]]\,\|_\infty,\\
    \overline{\alpha}(n) & := & \sup_{x\in\R}\,\|\,\pr[Y_{n}\in(-\infty,x]|{\cal F}^{0}](\cdot)-\pr[Y_{n}\in(-\infty,x]]\,\|_1,
\end{eqnarray*}
where ${\cal F}^{0} := \sigma(\{Y_{i}: i\leq 0\})$. In the following we will discuss that the strictly stationary and ergodic sequence $(X_i)_{i\in\N}$ can be extended to a strictly stationary sequence $(Y_{i})_{i\in\Z}$ being again ergodic and satisfying $\widetilde{\phi}(n) = \overline{\phi}(n)$ and $\widetilde{\alpha}(n) = \overline{\alpha}(n)$. More precisely, we may define a strictly stationary and ergodic sequence $(Y_{i})_{i\in\N_{0}}$ by $Y_{i} := X_{i+1}$, and Lemma 9.2 of \cite{Kallenberg1997} shows that this sequence can be extended to a strictly stationary sequence $(Y_{i})_{i\in\Z}$. Lemmas \ref{ergodic preservation} and \ref{Alternativrepraesentation} ahead show that $(Y_{i})_{i\in\Z}$ is again ergodic and that $\widetilde{\phi}(n) = \overline{\phi}(n)$ and $\widetilde{\alpha}(n) = \overline{\alpha}(n)$, $n\in\N$.

For any random variable $X$ on $\OFP$ and any sub-$\sigma$-algebra ${\cal A}\subseteq\cF$, the following dependence coefficients have been introduced in \cite{DedeckerPrieur2005}:
\begin{eqnarray}
    \phi({\cal A},X) & := &\sup_{x\in\R}\,\|\,\pr[X\in(-\infty,x]|{\cal A}](\cdot)-\pr[X\in(-\infty,x]]\,\|_\infty,
    \label{phi dc}\\
    \alpha({\cal A},X) & := & \sup_{x\in\R}\,\|\,\pr[X(-\infty,x]|{\cal A}](\cdot)-\pr[X\in(-\infty,x]]\,\|_1.
    \label{alpha dc}
\end{eqnarray}

\begin{lemma}\label{convergence dependency coefficients}
Let $({\cal A}_{i})_{i\in\N\cup\{\infty\}}$ be an increasing sequence of sub-$\sigma$-algebras of $\cF$ satisfying ${\cal A}_{\infty} = \sigma(\bigcup_{i\in\N} {\cal A}_{i})$. Then the following assertions hold:
\begin{itemize}
    \item[(i)] $\phi({\cal A}_{i},X)\leq\phi({\cal A}_{i+1},X)$ holds for $i\in\N$, and $\lim_{i\to\infty}\phi({\cal A}_{i},X) = \phi({\cal A}_{\infty},X)$.
    \item[(ii)] $\alpha({\cal A}_{i},X)\leq\alpha({\cal A}_{i+1},X)$ holds for $i\in\N$, and $\lim_{i\to\infty}\alpha({\cal A}_{i},X) = \alpha({\cal A}_{\infty},X)$.
\end{itemize}
\end{lemma}

\begin{proof}
Let us start by representations of the dependency coefficients established in \cite{DedeckerPrieur2005}. For this purpose let ${\rm BV}_{1}$ denote the space of all left continuous functions $f:\R\rightarrow\R$ with total variation bounded above by $1$. According to Lemmas 4 and 1 in \cite{DedeckerPrieur2005} we have
\begin{eqnarray}
    \phi({\cal A}_{i},X) & = & \sup\big\{|\covi(Y,h(X))|:\, Y\mbox{ is }{\cal A}_{i}\mbox{-measurable},\, \|Y\|_1\leq 1,\, h\in {\rm BV}_{1}\big\} \quad\label{phi representation}\\
    \alpha({\cal A}_{i},X) & = & \sup\big\{\|\ex\big[h(X)|{\cal A}_{i}\big] - \ex\big[h(X)\big]\|_1:\, h\in{\rm BV}_{1}\big\}
    \label{alpha representation}
\end{eqnarray}
for every $i\in\N\cup\{\infty\}$. We may observe immediately from (\ref{phi representation}) that
\begin{equation}\label{phi representation consequence}
    \phi({\cal A}_{i},X)\leq\phi({\cal A}_{i+1},X)\le\phi({\cal A}_{\infty},X)\quad\mbox{ for all $i\in\N$}.
\end{equation}
Furthmore, for any $h\in {\rm BV}_{1}$ and every $i\in\N$, we may observe
\begin{eqnarray*}
    \|\ex\big[h(X)|{\cal A}_{i}\big] - \ex\big[h(X)\big]\|_1
    & = & \ex\big[\big|\ex[(\ex[h(X)|{\cal A}_{i+1}] - \ex[h(X)])|{\cal A}_{i}]\big|\big]\\
    & \le & \ex\big[\ex[|\ex[h(X)|{\cal A}_{i+1}] - \ex[h(X)\big]|~|{\cal A}_{i}\big]\big]\\
    & = & \|\ex[h(X)|{\cal A}_{i+1}] - \ex[h(X)]\|_1.
\end{eqnarray*}
In view of (\ref{alpha representation}) this implies
\begin{equation}\label{alpha representation consequence}
    \alpha({\cal A}_{i},X)\leq\alpha({\cal A}_{i+1},X)\le\alpha({\cal A}_{\infty},X)\quad\mbox{ for all $i\in\N$}.
\end{equation}
For every fixed $h\in{\rm BV}_{1}$ we obtain by (\ref{alpha representation consequence}) and Theorem 10.5.1 in \cite{Dudley2002} (a version of Doob's martingale convergence theorem) that
$$
    \lim_{i\to\infty}\ex[h(X)|{\cal A}_{i}]=\ex[h(X)|{\cal A}_{\infty}]\qquad\pr\mbox{-a.s.}
$$
Since $h$ as an element of ${\rm BV}_1$ is bounded, it follows by the dominated convergence theorem that  
$$
    \lim_{i\to\infty}\|\ex[h(X)|{\cal A}_{i}] - \ex[h(X)|{\cal A}_{\infty}]\|_1=\lim_{i\to\infty}\ex\big[|\ex[h(X)|{\cal A}_{i}]-\ex[h(X)|{\cal A}_{\infty}]|\big]=0.
$$
For arbitrary $\varepsilon > 0$ we may find by (\ref{alpha representation}) some $h\in {\rm BV}_{1}$ such that the inequality $\alpha({\cal A}_{\infty},X) - \varepsilon < \|\ex[h(X)|{\cal A}_{\infty}] - \ex[h(X)]\|_1$ holds. Then by (\ref{alpha representation consequence}) along with (\ref{alpha representation})
\begin{eqnarray*}
    \alpha({\cal A}_{\infty},X)
    & \geq &
    \limsup_{i\to\infty}\alpha({\cal A}_{i},X)\\
    &\geq&
    \liminf_{i\to\infty}\alpha({\cal A}_{i},X)\\
    &\geq&
    \liminf_{i\to\infty}\|\ex\big[h(X)|{\cal A}_{i}\big] - \ex\big[h(X)\big]\|_1\\
    &=&
    \|\ex\big[h(X)|{\cal A}_{\infty}\big] - \ex\big[h(X)\big]\|_1\\
    &\geq&
    \alpha({\cal A}_{\infty},X) - \varepsilon.
\end{eqnarray*}
Hence $\lim_{i\to\infty}\alpha({\cal A}_{i},X)=\alpha({\cal A}_{\infty},X)$.  This completes the proof of statement (b).

Now, in addition of $h\in{\rm BV_{1}}$ let us fix any $\pr|_{{\cal A}_{\infty}}$-integrable random variable $Y$ with $\|Y\|_1=\ex[|Y|]\leq 1$. Firstly, $Y_{i} := \ex[Y|{\cal A}_{i}]$ and $Z_{i}:= \ex[|Y|\,|{\cal A}_{i}]$ define martingales adapted to the filtered probability space $(\Omega,({\cal A}_{i})_{i\in\N\cup\{\infty\}},{\cal A}_{\infty},\pr|_{{\cal A_{\infty}}})$ and they satisfy $\|Y_i\|_1=\ex[|Y_{i}|]\le 1$ as well as $\|Z_i\|_1= \ex[|Z_{i}|]\leq 1$ for every $i\in\N\cup\{\infty\}$. Hence we may draw on Theorem 10.5.1 of \cite{Dudley2002} again to observe $\lim_{i\to\infty}Y_{i}=Y$ $\pr$-a.s.\ and $\lim_{i\to\infty}Z_{i}=Z_{\infty}$ $\pr$-a.s.  Furthermore, $\lim_{i\to\infty}\ex[Z_{i}]=\ex[|Y|]$, and any $Z_{i}$ is nonnegative. Therefore, $(Z_{i})$ is uniformly $\pr$-integrable which implies that $(Y_{i})$ is uniformly $\pr$-integrable because $|Y_{i}|\leq Z_{i}$ $\pr$-a.s.\ holds for every $i\in\N\cup\{\infty\}$. Thus $\lim_{i\to\infty}\ex[|Y_{i} - Y|] =\lim_{i\to\infty}\|Y_{i} - Y\|_1=0$. In particular, $\lim_{i\to\infty}\ex[Y_{i} h(X)]=\ex[Y h(X)]$ since $h(X)$ is a bounded random variable. In particular $\covi(Y_{i},h(X))\to \covi(Y,h(X))$. If we now take (\ref{phi representation}), we may verify $\phi({\cal A}_{i},X)\to\phi({\cal A}_{\infty},X)$ in a similar way as we established $\alpha({\cal A}_{i},X)\to\alpha({\cal A}_{\infty},X)$. This shows the full statement (a) and completes the proof.
\end{proof}

Now let $(X_{i})_{i\in\N}$ be any strictly stationary sequence on $(\Omega,{\cal F},\pr)$, and let $(Y_i)_{i\in\Z}$ be the strictly stationary extension of $(X_{i})_{i\in\N}$ as introduced above. In the following lemma we describe the mixing coefficients $\widetilde{\phi}(n)$  and $\widetilde{\alpha}(n)$ in terms of the dependence coefficients defined in (\ref{phi dc})--(\ref{alpha dc}) and the sequence $(Y_{i})_{i\in\Z}$. Let $\Z_-:=\{0,-1,-2,\ldots\}$.

\begin{lemma}\label{Alternativrepraesentation}
For $n\in\N$ we have
$$
    \widetilde{\phi}(n) = \phi\big(\sigma(\{Y_{i}: i\in\Z_{-}\}),Y_{n}\big)\quad\mbox{ and }\quad\widetilde{\alpha}(n) = \alpha\big(\sigma(\{Y_{i}: i\in\Z_{-}\}),Y_{n}\big).
$$
\end{lemma}

\begin{proof}
Set ${\cal A}_{k} := \sigma(Y_{-k + 1},\ldots,Y_{0})$ for $k\in\N$, and ${\cal A}_{\infty}:= \sigma(\{Y_{i}: i\in\Z_{-}\})$. By strict stationarity of $(Y_{i})_{i\in\Z}$, the random vector $(Y_{-k + 1},\ldots,Y_{0},Y_{n})$ has the same distribution as $(Y_{0},\ldots,Y_{k-1},Y_{n + k - 1}) = (X_{1},\ldots,X_{k},X_{n+k})$ for every $k,n\in\N$. In particular, for every $k,n\in\N$, the random variables $Y_{n}$ and $X_{n + k}$ are identically distributed and $\pr_{Y_{n}|{\cal A}_{k}} = \pr_{X_{n + k}|\cF_{1}^{k}}$. Hence we may observe
$$
    \widetilde{\phi}(n) = \sup_{k\in\N}\phi({\cal A}_{k},Y_{n})\quad\mbox{ and }\quad\widetilde{\alpha}(n) = \sup_{k\in\N}\alpha({\cal A}_{k},Y_{n})\quad\mbox{ for every }k,n\in\N.
$$
Finally, ${\cal A}_{k}\subseteq{\cal A}_{k + 1}$ holds for $k\in\N$, and ${\cal A}_{\infty}$ is generated by $\bigcup_{k=1}^{\infty}{\cal A}_{k}$. Then the statement of Lemma \ref{Alternativrepraesentation} follows immediately from Lemma \ref{convergence dependency coefficients}.
\end{proof}

\begin{lemma}\label{ergodic preservation}
The sequence $(Y_{i})_{i\in\Z}$ is ergodic if the sequence $(X_{i})_{i\in\N}$ is ergodic.
\end{lemma}

\begin{proof}
For $I=\Z$ or $I=\N$, denote by $\cB(\R^{I})$ the standard Borel $\sigma$-algebra on $\R^{I}$ (generated by the standard product topology on $\R^{I}$) and let
$$
    S_{I}:\R^{I}\longrightarrow\R^{I},\qquad (x_{i})_{i\in I}\longmapsto (x_{i + 1})_{i\in I}
$$
be the (one-step) shift operator. Furthermore let ${\cal M}_{1}(S_{I})$ be the set of all probability measures $\mu$ on $\cB(\R^{I})$ satisfying $\mu = \mu\circ S_{I}^{-1}$, that is, the set of all probability measures $\mu$ on $\cB(\R^{I})$ under which the shift operator $S_I$ is measure-preserving. Recall that $\mu\in{\cal M}(S_I)$ is said to be {\em ($S_I$-) ergodic} if the corresponding invariant $\sigma$-algebra ${\cal I}$ (i.e.\ the set of all $A\in{\cal B}(\R^I)$ with $A=S_I^{-1}(A)$) is trivial (i.e.\ $\mu[A]\in\{0,1\}$ for all $A\in{\cal I}$). It is known that
\begin{equation}\label{ergodic measures as extreme points}
    \big\{\mu\in\cM_{1}(S_{I}):\, \mu\mbox{ is }S_{I}\mbox{-ergodic}\big\} = \big\{\mu\in\cM_{1}(S_{I}):\,\mu\mbox{ is extreme point of }~\cM_{1}(S_{I})\big\};
\end{equation}
see, for instance, Theorem 9.12 of \cite{Kallenberg1997}. The sequence $(Y_i)_{i\in\Z}$ is strictly stationary and so its distribution $\pr_{(Y_{i})_{i\in\Z}}$ belongs to $\cM_{1}(S_{\Z})$. Thus it suffices to show that $\pr_{(Y_{i})_{i\in\Z}}$ is an extreme point of $\cM_{1}(S_{\Z})$. This will be done by way of contradiction.

Suppose that $\pr_{(Y_{i})_{i\in\Z}}$ is not an extreme point of $\cM_{1}(S_{\Z})$. Then there exist different $\mu,\nu\in\cM_{1}(S_{\Z})$ and some $\lambda\in (0,1)$ such that $\pr_{(Y_{i})_{i\in\Z}} = \lambda\mu + (1-\lambda)\nu$. For the Borel-measurable mapping
$$
    \Pi_{\N}:\R^{\Z}\longrightarrow\R^{\N},\qquad (x_{i})_{i\in\Z}\longmapsto (x_{i})_{i\in\N}
$$
we have $S_{\N}\circ \Pi_{\N} = \Pi_{\N}\circ S_{\Z}$. Thus $(\widetilde{\mu}\circ\Pi_{\N}^{-1})\circ S_\N^{-1}=(\widetilde\mu\circ S_\Z^{-1})\circ\Pi_\N^{-1}=\widetilde\mu\circ\Pi_\N^{-1}$ for every $\widetilde{\mu}\in\cM_{1}(S_{\Z})$, and so $\widetilde{\mu}\circ\Pi_{\N}^{-1}\in\cM_{1}(S_{\N})$ for every $\widetilde{\mu}\in\cM_{1}(S_{\Z})$. Further, we clearly have $\pr_{(Y_{i})_{i\in\Z}}\circ\Pi_{\N}^{-1} = \pr_{(X_{i})}\circ S_{\N}^{-1}$ for the distribution $\pr_{(X_{i})}$ of $(X_{i})$, and by the stationarity of $(X_i)$ we also have $\pr_{(X_{i})}\in\cM_{1}(S_{\N})$, i.e.\ $\pr_{(X_{i})}\circ S_\N^{-1}=\pr_{(X_i)}$. Thus $\pr_{(Y_{i})_{i\in\Z}}\circ\Pi_{\N}^{-1}=\pr_{(X_{i})}$. In particular, $\pr_{(Y_{i})_{i\in\Z}}\circ\Pi_{\N}^{-1} = \lambda\, \mu\circ \Pi_{\N}^{-1} + (1 - \lambda)\, \nu\circ \Pi_{\N}^{-1}$. Since $\pr_{(Y_{i})_{i\in\Z}}\circ\Pi_{\N}^{-1}$ ($=\pr_{(X_i)}$) is an extreme point of $\cM_{1}(S_{\N})$ by assumption and (\ref{ergodic measures as extreme points}), it follows that $\mu\circ \Pi_{\N}^{-1} = \nu\circ \Pi_{\N}^{-1}$. Since $\mu$ and $\nu$ belong to $\cM_{1}(S_{\Z})$, we may conclude that $\mu$ and $\nu$ have identical finite-dimensional marginal distributions. In particular $\mu = \nu$, which contradicts the assumption $\mu\not=\nu$.
\end{proof}



\end{document}